\documentclass[11pt]{amsart}
\usepackage{amsmath,amssymb,amsfonts,amsthm,amssymb,amscd,amstext,amsxtra,amsopn,array,url,verbatim,mathrsfs,longtable}

\usepackage{amssymb,bbm}
\usepackage{color}
\usepackage{enumerate}
\usepackage{dsfont}
\usepackage{graphicx}
\usepackage{float}
\makeatletter
\@namedef{subjclassname@2010}{%
  \textup{2010} Mathematics Subject Classification}
\makeatother

\usepackage{mathtools}

\newtheorem{thm}{Theorem}[section]
\newtheorem{cor}[thm]{Corollary}
\newtheorem{lem}[thm]{Lemma}
\newtheorem{prop}[thm]{Proposition}

\theoremstyle{definition}

\numberwithin{equation}{section}

\newcommand{\R}{{\mathbb{R}}}

\renewcommand{\Re}{{\mathfrak{Re}}}
\renewcommand{\Im}{{\mathfrak{Im}}}

\renewcommand{\a}{\alpha}
\renewcommand{\b}{\beta}

\newcommand{\s}{\sigma}
\newcommand{\del}{\delta}
\begin{document}
%\today 
\baselineskip=17pt

\title[Explicit zero density for the Riemann zeta function]{Explicit zero density for the Riemann zeta function}

\author[H. Kadiri]{Habiba Kadiri}
\address{Department of Mathematics and Computer Science\\
University of Lethbridge\\
4401 University Drive\\
Lethbridge, Alberta\\
T1K 3M4 Canada}
\email{habiba.kadiri@uleth.ca}

\author[A. Lumley]{Allysa Lumley}
\address{Department of Mathematics and Statistics\\
York University\\
4700 Keele St\\
Toronto, Ontario\\
M3J 1P3
 Canada}
\email{alumley@yorku.ca}

\author[N. Ng]{Nathan Ng}
\address{Department of Mathematics and Computer Science\\
University of Lethbridge\\
4401 University Drive\\
Lethbridge, Alberta\\
T1K 3M4 Canada}
\email{nathan.ng@uleth.ca}

%\date{}
\thanks{
Research for this article is partially supported by the NSERC Discovery grants of H.K. (RGPIN-2015-06799) and N.N. (RGPIN-2015-05972).
The calculations were executed on the University of Lethbridge Number Theory Group Eudoxus machine, supported by an NSERC RTI grant.}

\begin{abstract}
Let $N(\sigma,T)$ denote the number of nontrivial zeros of the Riemann zeta function with real part greater than $\sigma$ and imaginary part between $0$ and $T$. We provide explicit upper bounds for $N(\sigma,T)$ commonly referred to as a zero density result. In 1937, Ingham showed the following asymptotic result  $N(\sigma,T)=\mathcal{O} ( T^{\frac83(1-\sigma)} (\log T)^5 )$.
Ramar\'{e} recently proved an explicit version of this estimate.
We discuss a generalization of the method used in these two results which yields an explicit bound of a similar shape while also improving the constants. 
\end{abstract}



\maketitle

\section{Introduction}
Throughout this article $\zeta(s)$ denotes the Riemann zeta function and $\varrho $ denotes a non-trivial zero of $\zeta(s)$
lying in the critical strip, $0 < \Re(s) < 1$.
Let $\frac12 < \s <1, T>0$, and define
\begin{equation}
  \label{def-NsigT}
  N(\sigma,T) = \# \{ \varrho =\beta+i\gamma : \  \zeta(\varrho )=0, 0 < \gamma < T \text{ and } \sigma < \beta < 1\}.
\end{equation}
We shall prove a non-trivial, explicit upper bound for $N(\sigma,T)$. Such a bound  is commonly referred to as a zero-density estimate. 
We denote RH the Riemann Hypothesis and $\text{RH}(H_0)$ the statement: 
\begin{equation}
  \label{partialRH}
  \text{RH}(H_0): 
\text{ all non-trivial zeros }  \varrho  \text{ of } \zeta(s)\ \text{ with }\ |\Im(\varrho )| \le H_0 \text{ satisfy } \Re(\varrho )=\frac12 . 
\end{equation}
Currently, the best published value of $H_0$ for which \eqref{partialRH} is true is due to David Platt \cite{Pla0}:
\[H_0=3.0610046 \cdot 10^{10}\] with $N(H_0)=103\,800\,788\,359$.
Other strong evidence towards the RH is the large body of zero-density estimates for $\zeta(s)$. 
Namely, very good bounds for $N(\sigma,T)$ in various ranges of $\sigma$.  
\newline
Let $\sigma >\frac{1}{2}$. 
In 1913 Bohr and Landau \cite{BohLan} showed that %there exists $a_1(\sigma)$ and $T_1(\sigma)$ such that
\begin{equation}
  \label{bohrlandau}
  N(\sigma,T) = \mathcal{O}\left( \frac{ T}{\sigma-\frac{1}{2}} \right) 
\end{equation}
for $T$ asymptotically large.
This result implies that for any fixed $\varepsilon >0$, almost all zeros of $\zeta(s)$ lie in the band 
$|\frac{1}{2}-\Re(s)| < \varepsilon$. 
This was improved in 1937 by Ingham \cite{Ing}, who showed
\begin{equation}
   \label{ingham}
  N(\sigma,T)  = \mathcal{O} \left( T^{(2+4c)(1-\s)} (\log T)^5 \right)  
\end{equation}
assuming that $\zeta(\frac12 +it) = \mathcal{O} \left( t^{c+\epsilon} \right) $. 
In particular, the Lindel\"of Hypothesis $\zeta(\frac12 +it) = \mathcal{O} \left( t^{\epsilon} \right) $ implies that
$  N(\sigma,T) = \mathcal{O} \left( T^{2(1-\s)+\epsilon}\right) $, also known as the Density Hypothesis. 
There is a prolific literature on the bounds for $\zeta(s)$, starting with the convexity bound of 
$c=\frac14=0.25$ (Lindel\"of), the first subconvexity bound of Hardy \& Littlewood \cite{HaLi} $c=\frac16=0.1666\ldots$, to some more recent results 
of Huxley \cite{Hux} (2005) $c=\frac{32}{205} = 0.1560\ldots$ and of Bourgain \cite{Bou} (2017) $c=\frac{13}{84} = 0.1547\ldots$.
In addition, there are also many articles on estimates for $N(\s, T)$.  A selection of some notable results may be found in \cite{Hux}, \cite{Hux2}, \cite{Ju}, and \cite{Bou}. 
On the other hand, there are few explicit bounds for $N(\sigma,T)$.
We refer the reader to a result of the first author \cite{Kad} for an explicit version of Bohr and Landau's bound.
The method provides two kind of results:
for $T$ asymptotically large, as in
$
N( 0.90,T) \le 0.4421 T + 0.6443 \log T - 363\,301,
$
and for $T$ taking a specific value, as in
$N(0.90,H_0) < 96.20$.
These bounds are useful to improve estimates of prime counting functions, as in \cite{FabKad}, \cite{Dus4}, \cite{PlaTru}, \cite{Tru} and in \cite{KadLum} to find primes in short intervals.
Ramar\'{e}  had earlier proven a version of \eqref{ingham} in his D.E.A. memoire,
which remained unpublished until recently. 
Let $\s \ge 0.52$ be fixed.
In \cite{Ram} he proves \footnote{Equation (1.1)  \cite[p. 326 ]{Ram} gives the bound  $N(\s,T) \le 4.9 
 (3T)^{\frac{8(1-\s)}{3}} (\log T)^{5-2\s} + 51.5 (\log T)^2$.  However,  there is a mistake in \cite{Ram}.  
 The authors have been in communication with Professor Ramar\'{e} and he has sent us 
 a proof of the revised inequality  \eqref{ram-NsigT-eq2}.  
  } that 
for any $T\ge 2000$
\begin{equation} \label{ram-NsigT-eq2}
N(\s,T) \le  
%4.9 
965
 (3T)^{\frac{8(1-\s)}{3}} (\log T)^{5-2\s} + 51.5 (\log T)^2 ,
\end{equation}
which gives 
$
N(0.90,T) < 1293.48   (\log T)^{\frac{16}{5}}  T^{\frac4{15} } + 51.50 (\log T)^2 ,
$
which gives the bound for $T=H_0$:
$
N(0.90, H_0)  < 2.1529\cdot10^{10}.
$
The purpose of this article is to bound $N(\s,T)$ by applying Ingham's argument with a general weight 
and to improve both \cite{Kad} and \cite{Ram}. 
\begin{thm} \label{thm-NsigT}
Let  $\frac{10^9}{H_0} \le k \le 1,d>0, H\in [1002,H_0)$, $\a>0$, $\del\ge1$, $\eta_0=0.23622 \ldots$, $1+\eta_0\le \mu \le 1+\eta$,
and $\eta\in(\eta_0,\tfrac12)$ be fixed. Let $\s > \frac12 +\frac{d}{\log H_0}$.\\
Then there exist $\mathcal{C}_1 ,\mathcal{C}_2 >0$ such that, for any $T\ge H_0$,
\begin{equation}
 \label{KLN-NsigT1}
N(\s,T)
\le \frac{(T-H)(\log T)}{2\pi d}
\log
\Big( 1 
+\frac{  \mathcal{C}_1
(\log (kT))^{2\s}  (\log T)^{4(1-\s)} T^{\frac83(1-\s)}
 }{T-H}  \Big)
+ \frac{ \mathcal{C}_2  }{2\pi d} (\log T)^2 ,
 %\text{ and}\ 
\end{equation}
where $\mathcal{C}_1= \mathcal{C}_1(\alpha, d,\del, k, H, \s)$ and $\mathcal{C}_2=\mathcal{C}_2 (d, \eta, k, H, \mu,\s)$ are defined in \eqref{def-mathcalC1} and \eqref{def-mathcalC2}.
Since $\log(1+x) \le x$ for $x \ge 0$, \eqref{KLN-NsigT1} implies
\begin{equation}
\label{KLN-NsigT2}
N(\s,T)
\le 
 \frac{ \mathcal{C}_1}{2\pi d}
(\log (kT))^{2\s}  (\log T)^{5-4\s } T^{\frac83(1-\s)}
+ \frac{ \mathcal{C}_2  }{2\pi d} (\log T)^2.
 \end{equation}
In addition, numerical results are displayed in tables in Section \ref{tables}. %Tables \ref{g2asymptotic} and \ref{g1optatH0}.
\end{thm}
For instance \eqref{KLN-NsigT2} gives
$
N(0.90,T) < 11.499 (\log T)^{\frac{16}{5}}  T^{\frac4{15} } + 3.186 (\log T)^2 ,
$
and \eqref{KLN-NsigT1} gives
$
N(0.90,H_0) < 130.07.
$
This improves previous results both numerically and methodologically (one of the key ingredients is the choice of a more efficient weight function in Ingham's method). 
Note that choosing $k<1$ and optimizing in $H$ can provide extra improvements to \eqref{ram-NsigT-eq2}. % and \eqref{ram-NsigT-eq1}. 
In addition, we prove a stronger bound for the argument of a holomorphic function.  
We now explain the main ideas to prove Theorem \ref{thm-NsigT}.
\section{Setting up the proof}
\subsection{Littlewood's classical method to count the zeros}
% %%
Let $h(s)=\zeta(s)M(s)$ where $M(s)$ is entire and 
\begin{equation}
   \label{NhsigT}
  N_h(\sigma,T) = \# \Big\{ \varrho '=\beta'+i\gamma' \in \mathbb{C} \ : \  h(\varrho ')=0,  \sigma < \beta' < 1, \text{ and }
  0 < \gamma' < T   \Big\}.
\end{equation}
Then for a parameter $H\in(0, H_0)$, we have by \eqref{partialRH} that
\[
   N(\s,T) =N(\sigma,T)-N(\sigma,H) \le N_h(\sigma,T)-N_h(\sigma,H)
\]
for $T \ge H_0$.
We compare the above number of zeros for $h$ to its average:
\[
   N_h(\sigma,T)-N_h(\sigma,H) \le \frac{1}{\sigma-\sigma'}  \int_{\s'}^{\mu} 
   ( N_h(\tau,T)-N_h(\tau,H)) \, d \tau
\]
where $\mu > 1$ and $\sigma'$ is a parameter satisfying $\frac12 < \sigma' < \sigma$. 
Let $\mathcal{R}$ be the rectangle with vertices $\sigma'+iH$, $\mu+iH$, $\mu+iT$, and $\sigma'+iT$.
We apply the classical lemma of Littlewood as stated in \cite[(9.9.1)]{Tit}:
\begin{equation}\label{approxineq} 
 \int_{\s'}^{\mu}\Big( N_h(\tau,T) -N_h(\tau,H) \Big) d\tau  = - \frac1{2\pi i} \int_{\mathcal{R}} \log h(s) ds . 
\end{equation}
Thus
\begin{multline}\label{boundN2}
N(\s,T)
\le   \frac{1}{2\pi(\s-\s')}  \Big( 
\int_{H}^{T} \log| h(\s'+it)| dt
\Big. \\ \Big.+    \int_{\s'}^{\mu} \arg h(\tau+iT) d\tau 
-  \int_{\s'}^{\mu} \arg h(\tau+iH) d\tau 
  - \int_{H}^{T} \log | h(\mu+it)| dt
\Big).
\end{multline}
As $T$ grows larger, the main contribution arises from the first integral.
The second and third integrals can be treated by using a general result for bounding
$\arg f(s)$ for $f$ a holomorphic function.  To do this we give an improvement of  
a lemma of Titchmarsh \cite[p. 213]{Tit}  (see Proposition \ref{Backlund} and Corollary \ref{arghX} below).
The fourth integral can be estimated with 
a standard mean value theorem for Dirichlet polynomials (see Lemma \ref{ExplicitMV}).
A key goal is to minimize the above expression over admissible functions $h$.
We now give an idea of how to estimate the first integral in \eqref{boundN2}. 
\subsection{How the second mollified moment of $\zeta(s)$ occurs}
 Let $X \ge 1$ be a parameter and define the mollifier to be 
\begin{equation}
   \label{def-Mxs}
  M_X(s) =  \sum_{n \le X} \frac{\mu(n)}{n^s}
\end{equation}
where $\mu(n)$ is the M\"{o}bius function.   Note that this is a truncation of the Dirichlet series for $\zeta(s)^{-1}$.
These mollifiers were invented by Bohr and Landau \cite{BohLan} to help control the size of $\zeta(s)$ in the critical strip.
Futhermore, let 
\begin{equation}
  \label{def-fXs}
  f_X(s) = \zeta(s) M_X(s) -1.
\end{equation}
Note that the series expansion for $f_X$ is given by
\begin{align}
& \label{series-fX}
  f_X(s) 
= \sum_{n > X}  \Big( \sum_{ \substack{ d \mid n \\  d \le X}} \mu(d) \Big) n^{-s}
= \sum_{n\ge 1} \frac{\lambda_X(n)}{n^s} ,
%= \sum_{n=1}^{\infty} \Big( \sum_{ \substack{ d \mid n \\  d \le X}} \mu(d) \Big) n^{-s}-1.
\\
\text{with }\ 
&
  \label{lambdaXdefn}
  \lambda_X(n) 
= 
0 \ \text{ if } n \le X,\ \ 
\lambda_X(n) = \sum_{\begin{substack} {d \mid n \\ d \le X} \end{substack}} \mu(d) \ \text{ if } n > X. 
\end{align}
We shall choose $h = h_X$ with
\begin{equation}
  \label{def-hXs}
  h_X(s) = 1-f_X(s)^2=\zeta(s)M_X(s) (2-\zeta(s)M_X(s)). 
\end{equation}
Since we have \[ \frac1{b-a} \int_a^b \log f(t) dt \le  \log \left( \frac1{b-a} \int_a^b f(t) dt\right),\] for any $f$ non-negative and continuous,
and  $| h_X(s)|  \le 1 + |f_X(s)|^2$,
we deduce that
\begin{equation}
   \label{Jensenbnd}
    \int_{H}^{T}  \log \left(  | h_X(\s'+it)| \right) dt 
\le (T-H) \log\left( 1  +  \frac1{T-H}  \int_{H}^{T} |f_X(\sigma'+it)|^2 dt \right).
\end{equation}
We denote 
\begin{equation}
  \label{def-FX}
   F_X(\sigma,T) = \int_{0}^{T} |f_X(\sigma+it)|^2 dt  \text{ where } \sigma \ge \frac{1}{2}.
\end{equation}
To resume, the key point for getting a good bound on $N(\sigma,T) - N(\sigma,H)$ is to obtain a good
bound for $F_X(\sigma,T)$. Following a classical method due to Ingham we compare it to a smoothed version of itself.
\subsection{Ingham's smoothing method}\label{Ingham-method}
Let $\sigma_1$ and $\sigma_2$ be such that $\s_1 < \sigma < \sigma_2$.
Let $T>0$ and $g=g_T$ be a non-negative, real valued function, depending on the parameter $T$, and holomorphic in $\sigma_1 \le \Re(s) \le \sigma_2$.
We define
\begin{equation}\label{def-Mgsigma}
 \mathcal{M}_{g,T} (X, \s)  = \int_{-\infty}^{+\infty} |g (\sigma+it)|^2 |f_X(\sigma+it)|^2 dt .
\end{equation}
We shall consider $g $ of a special shape.  For $\a,\beta>0$, assume that there exist positive functions $\omega_1,\omega_2$ such that $g$ satisfies, for all $\sigma \in [\sigma_1,\sigma_2] $,
\begin{align}\label{gupbd}
&   
 |g (\sigma+it)|\le \omega_1(\s,T,\a)e^{-\a\left(\frac{|t|}{T}\right)^{\beta}}\ \text{ for all}\  t,
\\
\label{glbd}
&   \omega_2(\sigma,T,\a) \le |g (\sigma+it)|   \ \text{ for all}\  t\in[H,T].
\end{align}
In addition, we assume that $|g|$ is even in $t$: 
\begin{equation}
 \label{gabseven}
 |g(\s-it)|=|g(\s+it)| \text{ for } \s \in (\s_1,\s_2) \text{ and } t \in \mathbb{R}.
\end{equation}
Thus $F_X(\sigma,T) \ll_g  \mathcal{M}_{g,T} (X, \s) $, and more precisely
\begin{equation}\label{bnd-F}
F_X(\s,T)  \le \frac{\mathcal{M}_{g,T} (X, \s) }{2(\omega_2(\sigma,T,\a))^2 }  .     
\end{equation}
In this article, we shall choose a family of weights of the form 
\begin{equation}
  \label{gdefn}
  g(s) =  g_{T}(s) = \frac{s-1}{s} e^{\alpha (\frac{s}{T})^2 },  \text{ where } \alpha >0.
\end{equation}
These weights will satisfy the above conditions %\eqref{glbd-gupbd}  
with $\beta=2$.
We remark that Ingham \cite{Ing} made use of the weight $g(s) = \frac{s-1}{s\cos(\frac{1}{2T})}  $ and Ramar\'{e} \cite{Ram}  used 
$g(s)= \frac{s-1}{s(\cos s )^{\frac{1}{2T}}} $.  These weights satisfy \eqref{gupbd} with $\beta=1$.  We also 
studied the weights $g(s) = \frac{s-1}{s(\cos s )^{\frac{\alpha}{T}}}$
and  $g(s) = \frac{s-1}{s(\cos \frac{\alpha}{T} )}$.  
However, we obtained the best results with $g$ given by \eqref{gdefn}. 
The functions $g$ are chosen so that for fixed $\sigma$,   $g(\sigma+it)$ behave likes the indicator function, $\mathds{1}_{[0,T]}(t)$,
and for $t$ large, $g(\sigma+it)$ has rapid decay. 
Nevertheless, it is an open problem to determine the best weights $g$ to use in this problem.   
\subsection{Final bound}
Finally, to bound the integral $\mathcal{M}_{g,T} $, we appeal to a convexity estimate for integrals (see \cite{HIP}).
For $\sigma_2 >1$ (and $\s_2$ close to $1$), if $\frac12 \le \s \le \s_2$, then
\begin{equation}
\mathcal{M}_{g,T} (X, \s) \le \mathcal{M}_{g,T} (X, \tfrac12 )^{\frac{\sigma_2-\sigma}{\sigma_2-\frac12 }}  
\mathcal{M}_{g,T} (X, \sigma_2)^{\frac{\sigma-\frac12 }{\sigma_2-\frac12 }}.
\end{equation}
The largest contribution arises from $ \mathcal{M}_{g,T} (X, \frac12 )$. 
To bound this we make use of:
\begin{itemize}
\item bounds \eqref{gupbd}, \eqref{glbd}  for $g$ (see Lemma \ref{UpboundPhicor}), 
\item a version of Montgomery and Vaughan's Mean Value Theorem for Dirichlet polynomials (see Lemma \ref{ExplicitMV}), 
\item bounds for arithmetic sums to bound the second moment of the mollifier $M_X$ (we use Ramar\'e's bounds, see Lemma \ref{mu2bd} and \ref{lambdaXbd}),
\item the most recent explicit subconvexity bound for the Riemann zeta function (due to Hiary \cite{Hiary}, see Lemma \ref{zetahalf}).
\end{itemize}
\section{Preliminary lemmas}
\subsection{Bounds for the Riemann zeta function}
In this section we record a number of bounds for the zeta function. 
Rademacher \cite[Theorem 4]{Rad} established the following explicit convexity bound. 
\begin{lem}
\label{bndLschi}
For $-\frac{1}{2}\le -\eta\le \sigma \le 1+\eta\le\frac32$, we have 
\begin{equation}
|\zeta(s)|\le 3\frac{|1+s|}{|1-s|}\left(\frac{|1+s|}{2\pi}\right)^{\frac{1}{2}(1-\sigma+\eta)}\zeta(1+\eta).
\end{equation}
\end{lem}
The next lemma is an explicit version of van der Corput's subconvexity bound for $\zeta$ on the critical line, recently proven by Hiary.  
\cite{Hiary}. 
\begin{lem}\label{zetahalf}  
We have 
\begin{align}
&  \label{zetaonehalfbd}
  \ |\zeta(\tfrac{1}{2}+it)|\le a_1 t^{\frac16}\log t &\ \text{for all }\ t\ge 3,\\
&
  \label{maxzeta}
\max_{|t|\le T}|\zeta(\tfrac{1}{2}+it)| \le a_1 T^{\frac16}\log T+a_2   &\ \text{for all }\ T>0,
 \end{align}
 with
 \begin{equation}\label{def-a1-a2}
a_1  = 0.63 \text{ and } a_2  = 2.851.
\end{equation}
\end{lem}
\begin{proof}[Proof of Lemma \ref{zetahalf}]
Statement \eqref{zetaonehalfbd} is \cite[Theorem 1.1]{Hiary}. 
For $T \in [0,3]$, \cite[Theorem 1.1]{Hiary} provides that $|\zeta(\tfrac12 +it)|\le 1.461$. 
We find that the minimum of the function $t^{\frac16}\log(t)$ occurs when $t=e^{-6}$. We require the polynomial $a_1t^{\frac16}\log(t)+a_2\ge1.461$, choosing $a_2$ as in the statement of the lemma achieves this.
\end{proof}
\subsection{Bounds for arithmetic sums}
We list here some preliminary lemmas from \cite{Ram} providing estimates for finite arithmetic sums.
Let
\begin{equation} \label{def-bi}
\begin{split}
b_{1} = 0.62 ,  \ 
b_{2}  = 1.048,  \ 
b_{3}  = 0.605 , \ \text{and}\ 
b_{4}  = 0.529.  
\end{split}
\end{equation}
\begin{lem}  \label{mu2bd}
We have
\begin{align}
& \label{bound-sum-mu2}
 \sum_{n \le X} \mu^2(n) \le
    b_{1} X \ \text{ for all }\ X \ge 1700,
\\&
\label{bound2-sum-mu2/n} 
\sum_{n \le X} \frac{\mu^2(n)}{n} -\frac{6}{\pi^2}  \log X  \le b_{2}     \ \text{ for all }\ X\ge 1002.
\end{align}
\end{lem}
\eqref{bound-sum-mu2} is \cite[Lemma 3.1]{Ram} %, \eqref{bound1-sum-mu2/n}
 and \eqref{bound2-sum-mu2/n} is \cite[Lemma 3.4]{Ram}. 
\begin{lem} \label{lambdaXbd}
Let $\tau >1, \del >0$, $X \ge 10^9$, and  $\gamma$ denotes Euler's constant. 
Then
\begin{align}
&   \label{lambdaXn5X}
  \sum_{X < n < 5X} \frac{\lambda_X(n)^2}{n^2} \le 
  \frac{b_{3} }{X}, 
\\
&
  \label{lambdaXn1}
  \sum_{n \ge1} \frac{\lambda_X(n)^2}{n^{\tau}} \le \frac{b_{4}  \tau^2}{\tau-1} e^{\gamma(\tau-1)} \log X,
\\&
 \label{lambdaXntau1}
  \sum_{n \ge1} \frac{\lambda_X(n)^2}{n^{1+\frac{\del }{\log X}}} \le \frac{b_{4} }{\del } 
  \Big(1+\frac{\del }{\log X} \Big)^2e^{\frac{\del \gamma}{\log X}} (\log X)^2 ,  
\\&
  \label{lambdaXntau}
  \sum_{n \ge1} \frac{\lambda_X(n)^2}{n^{2+\frac{2\del }{\log X}}} 
\le 
\frac{b_{4} }{5\del  e^{\del} } 
  \Big(1+\frac{{\del} }{\log X} \Big)^2 e^{\frac{{\del} (\gamma-\log 5)}{\log X}} \frac{(\log X)^2}{ X}
  +
\frac{b_{3} e^{-2{\del} }}{X}.
\end{align}
\end{lem}
\begin{proof} 
\eqref{lambdaXn5X} is \cite[Lemma 5.6]{Ram} and 
\eqref{lambdaXn1} is \cite[Lemma 5.5]{Ram}.
\eqref{lambdaXntau1} is a direct consequence of \eqref{lambdaXn1}, taking $\tau=1+\frac{{\del} }{\log X}$. 
\newline
For \eqref{lambdaXntau} we set $\tau=2+\frac{2{\del} }{\log X}$. Since $\lambda_X(n)^2=0$ when $1\le n \le X$, 
then
$$
    \sum_{n \ge1} \frac{\lambda_X(n)^2}{n^{\tau}}
    = \sum_{X < n < 5X} \frac{\lambda_X(n)^2}{n^{\tau}} +
    \sum_{n \ge 5X} \frac{\lambda_X(n)^2}{n^{\tau}}.
$$
Since $\tau \ge 2$, we use \eqref{lambdaXn5X} and find that the first sum is 
$$
   \le \frac{1}{X^{\tau-2}}  \sum_{X < n < 5X} \frac{\lambda_X(n)^2}{n^2}  \le 
    \frac{1}{X^{\tau-2}} \frac{b_{3} }{X} %= \frac{b_{3} }{X^{\tau-1}}
    =\frac{b_{3} e^{-2{\del} }}{X}. 
$$
We bound the second sum using $n^{\tau} \ge (5X)^{1+\frac{{\del} }{\log X}} n^{1+\frac{{\del} }{\log X}}$ and \eqref{lambdaXntau1}. We find that it is
$$
    \le \frac{1}{(5X)^{1+\frac{{\del} }{\log X}}}    \frac{b_{4} }{{\del} } \Big(1+\frac{{\del} }{\log X} \Big)^2 e^{\frac{{\del} \gamma}{\log X}} (\log X)^2.
$$
Combining bounds
\[
 \sum_{n \ge1} \frac{\lambda_X(n)^2}{n^{\tau}} 
      \le \frac{b_{4} }{5{\del}  e^{\del} } \Big(1+\frac{{\del} }{\log X} \Big)^2 e^{\frac{{\del} (\gamma-\log 5)}{\log X}} \frac{(\log X)^2}{ X} 
+ \frac{b_{3} e^{-2{\del} }}{X}.
\]
\end{proof}
\begin{lem}\label{divisorsum} Let $\tau>1$ and $\gamma$ is Euler's constant.  Then for $X \ge 1$,
\begin{equation}
  \label{divisorsum1}
\sum_{n\ge X}\frac{d(n)}{n^\tau}\le\frac{\tau}{X^{\tau-1}}\left(\frac{\log X}{\tau-1}+\frac1{(\tau-1)^2}+\frac{\gamma}{\tau-1}+\frac7{12\tau X}\right) 
\\%\text{and}\ 
\end{equation}
and for $X \ge 47$, 
\begin{equation}
  \label{divisorsum2}
\sum_{n\ge X}\frac{d(n)^2}{n^\tau}\le \frac{2 \tau}{X^{\tau-1}} \Big(
\frac{(\log X)^3}{\tau-1}+ \frac{3 \log^2 X}{(\tau-1)^2}+ \frac{6 \log X}{(\tau-1)^3}
  + \frac{6}{(\tau-1)^4} \Big).
\end{equation}
\end{lem}
\begin{proof}
By partial summation, we have
\begin{align*}
\sum_{n\ge X}\frac{d(n)}{n^{\tau}}\le \tau\int_X^{\infty}\frac{\sum_{n \le t} d(n)}{t^{\tau+1}}dt.
\end{align*}
Using $\sum_{n \le t} d(n) \le t(\log t+\gamma+\frac{7}{12t})$, for $t \ge 1$, which follows from  \cite[Equation 3.1]{Ram0}, we have
\begin{align*}
\sum_{n\ge X}\frac{d(n)}{n^{\tau}}\le \tau \left(\int_X^{\infty}\frac{\log t}{t^{\tau}}dt+\gamma\int_X^{\infty}\frac{dt}{t^{\tau}}+\frac7{12}\int_X^{\infty}\frac{dt}{t^{\tau+1}}\right).
\end{align*}
By applying the integrals
\[\int_X^{\infty}\frac{\log t}{t^c}dt=\frac{\log X}{(c-1)X^{c-1}}+\frac{1}{(c-1)^2X^{c-1}} \text{ and } \int_X^{\infty}\frac{dt}{t^c}=\frac1{(c-1)X^{c-1}}, \text{ where }  c>1,\]
we obtain \eqref{divisorsum1}.   The second estimate is similar.  We have 
\[\sum_{n\ge X}\frac{d(n)^2}{n^\tau}\le \tau \int_{X}^{\infty} \frac{ \sum_{n \le t} d(n)^2}{t^{\tau+1}} \, dt.
\]   
It suffices to use the elementary bound $ \sum_{n \le t} d(n)^2 \le t (\log t + 1)^3  
\le 2 t \log^3 t$ for $t \ge 47$, derived by Gowers \cite{Go}. 
Thus 
\[
  \sum_{n\ge X}\frac{d(n)^2}{n^\tau}\le 2\tau \int_{X}^{\infty} \frac{\log^3 t}{t^{\tau}} \, dt
  = 2 \tau \left( 
  \frac{\frac{(\log X)^3}{\tau-1}+ \frac{3 \log^2 X}{(\tau-1)^2}+ \frac{6 \log X}{(\tau-1)^3}
  + \frac{6}{(\tau-1)^4}}{X^{\tau-1}}
  \right).
\]
%and we are finished.
\end{proof}

\subsection{Mean value theorem for Dirichlet polynomials}
We require Montgomery and Vaughan's mean value theorem for Dirichlet polynomials 
in the form derived by Ramar\'e \cite{Ram}. 
\begin{lem}\label{ExplicitMV} 
Let $(u_n)$ be a real-valued sequence.
For every $T \ge 0$ we have
\begin{equation}\label{bound-MV-0T}
\int_0^T \Big| \sum_{n=1}^{\infty} u_nn^{it} \Big|^2dt\le \sum_{n \ge1}|u_n|^2(T+\pi m_0(n+1)), 
\end{equation}
with 
\begin{equation}\label{def-m0}
m_0=\sqrt{1+\frac23 \sqrt{\frac{6}{5}}}.
\end{equation}
Let $0 < T_1 < T_2$. Then 
\begin{equation}\label{bound-MV-T1T2}
\int_{T_1}^{T_2} |\sum_{n=1}^{\infty} u_nn^{it}|^2dt\le \sum_{n \ge1}|u_n|^2(T_2-T_1+2 \pi m_0(n+1)). 
\end{equation}
%Moreover, when $u_n$ is real-valued, the constant $2\pi m_0$ may be reduced to $\pi m_0$.
\end{lem}
\begin{proof}
The inequality \eqref{bound-MV-0T} is \cite[Lemma 6.5]{Ram}, and \eqref{bound-MV-T1T2} follows by the same proof.
This argument is an explicit version of Corollary 3 of \cite{MV} which makes use of the main theorem of \cite{Pr}.
Note that \eqref{bound-MV-T1T2} follows from two applications of \eqref{bound-MV-0T}.
\end{proof}
\subsection{Choice for the smooth weight $g$}
\begin{lem}%[Corollary of Lemma \ref{UpboundPhi}]
\label{UpboundPhicor}
Let $ \alpha >0$ and $\beta=2$. Let $s=\s+it$ and let $g$ be as defined in \eqref{gdefn}:
\begin{equation}
g(s) = \frac{s-1}{s}e^{\a\left(\frac{s}{T}\right)^{2}}.                                
\end{equation}
 %Let $0<\alpha<1$. Assume $2< T^{\frac13 }/\alpha<H \le |t| \le T, \frac12 <\s<1$. 
 %We define
 %\begin{equation}\label{def-g3}
 %g (s)=\frac{s-1}{s}e^{\a\left(\frac{s}{T}\right)^{2}}. %, \text{ for } \frac12\le \s \le 1+d. 
 %\end{equation}
Let  $\sigma_1=\frac12, \sigma_2>1$, and $H<T$. 
Define 
\begin{eqnarray}
& \label{def-omega1} \omega_1(\s,T,\a) &= e^{\a\left(\frac{\s}{T}\right)^2},\\
& \label{def-omega2} \omega_2(\sigma,T,\a)&=\left(1-\frac1{H}\right)e^{\a\left(\frac{\s}{T}\right)^2-\a}.
\end{eqnarray}
Then for $\frac12 \le \sigma \le \sigma_2$, $g$ satisfies \eqref{gupbd} and \eqref{glbd}:
\begin{align}
\label{expl-gupbd}
& |g (\sigma+it)| \le \omega_1(\s,T,\a)e^{-\a\left(\frac{|t|}{T}\right)^{2}}  & \text{ for  all } \ t,
\\ \label{expl-glbd}  \omega_2(\sigma,T,\a) \le & |g (\sigma+it)|  & \text{ for }\ H \le t \le T.
\end{align}
\end{lem}
\begin{proof} 
Since  $\s\ge \frac12 $, we have
$ 
\left|\frac{s-1}{s}\right|^2= 1-\frac{2\s-1}{\s^2+t^2} \le 1 .
$ 
Thus $
|g (s)|\le | e^{\a(\frac{s}{T})^2} | %= \left| e^{\frac{\a}{T^2}(\s^2+2i\s t-t^2)} \right| 
= e^{\frac{\a\s^2}{T^2}}e^{\frac{-\a t^2}{T^2}}$ and we have the expression for $\omega_1(\s,T,\a)$.  \\
In addition,
$  |\frac{s-1}{s}|=|1-\frac1{s} |\ge 1-\frac{1}{|s|} \ge 1-\frac1{|t|}$, so for all  $t \in [H,T]$, we have
\[
|g (s)|  \ge  \left(1-|t|^{-1}\right) e^{\frac{\a\s^2}{T^2}}e^{\frac{-\a t^2}{T^2}}  \ge (1-H^{-1})  e^{\frac{\a\s^2}{T^2}}  e^{-\a},
\]
which gives $\omega_2(\s,T,\a)$. 
\end{proof}
\section{Proof of the Main Theorem}
Unless specified in the rest of the article, we set $ H_0=3.0610046 \cdot 10^{10}$ 
and we have the following conditions on the parameters $k, \sigma_1, \delta$, and $\s_2$:
\begin{align}
 \label{def-parameters}
 & k\ge \frac{10^9}{H_0}, \  \sigma_1 = \frac12,\ \delta>0,\ \text{and}\ \s_2=1+\frac{{\del} }{\log X}.
\end{align}
%%%
\subsection{Bounding  $F_X(\s,T)$}
We establish here some preliminary lemmas to estimate $F_X(\s,T)$ at $\frac{1}{2}$ and at $1+\frac{{\del} }{\log X}$. 
%%%
\subsubsection{Bounding $F_X(\frac12,T)$}
We first need to bound the second moment of $M_X(\frac12+it)$, where $M_X$ is defined in \eqref{def-Mxs}.
\begin{lem}\label{bndMX1/2}
Let $T > 0$, $X \ge kH_0$, and $k$ satisfies \eqref{def-parameters}.  
Then
\begin{equation}\label{bound-moment-M-1/2}
 \int_0^T\left|M_X(\tfrac12+it)\right|^2dt\le(C_1T+C_2X) (\log X),
\end{equation}
where
\begin{align}
 \label{def-C1} C_1 & = C_1(k) =
\frac{6}{\pi^2} 
 +  \frac{b_{2}}{\log (kH_0)},\\
\label{def-C2} C_2 & = C_2(k) =
\frac{\pi m_0 b_{1} }{\log (kH_0)}+ \frac{6 m_0}{\pi kH_0} 
 + \frac{\pi m_0 b_{2}}{kH_0 \log (kH_0)},
\end{align}
and the $b_i$'s are defined in \eqref{def-bi} and $m_0$ in \eqref{def-m0}.
\end{lem}
\begin{proof}
We apply \eqref{bound-MV-0T} to $u_n = \frac{\mu(n)}{n^{\frac12 }}$:
\[\int_0^T\left|M_X\left(\tfrac12+it\right)\right|^2dt\le \sum_{n\le X}\frac{\mu^2(n)}n(T + \pi m_0(n+1)).\]
Since $X\ge 1700$, we apply \eqref{bound-sum-mu2} to $(T + \pi m_0) \sum_{n\le X}\frac{\mu^2(n)}n $
and \eqref{bound2-sum-mu2/n}  to $(\pi m_0) \sum_{n\le X} \mu^2(n) $ respectively.
We factor $\log X$ to give 
\begin{align*}
\int_0^T\left|M_X\left(\tfrac12+it\right)\right|^2dt
& =
(T + \pi m_0) \left(\frac6{\pi^2}\log X+b_{2}  \right)
+ \pi m_0 b_{1} X
\\ & =
\left( \left( \frac{6}{\pi^2} 
+  \frac{b_{2}}{\log X}\right) T
+\left( \frac{6m_0}{\pi X} 
+ \frac{\pi m_0 b_{2}}{X \log X}
+   \frac{\pi m_0 b_{1} }{\log X}\right) X 
\right) (\log X),
\end{align*}
and use the fact that $X  \ge kH_0$ to obtain the announced bound. 
\end{proof}
%%%
\begin{lem}\label{bndFqhalf}
Let $T > 0$, $X \ge kH_0$, and $k$ satisfies \eqref{def-parameters}. Then
\begin{equation}\label{bound-momentF-1/2}
F_X(\tfrac{1}{2},T)\le 
C_4\left(T^{\frac16}\log T+\frac{a_2}{a_1}\right)^2\left(T+\frac{C_2}{C_1}X\right) (\log X),
\end{equation}
where $a_1,a_2$ are defined in \eqref{def-a1-a2}, $C_1$ in \eqref{def-C1}, $C_2$ in \eqref{def-C2}, and
\begin{align}
& \label{def-a3}a_3=-\frac{6a_1}{e}+a_2,\\
& \label{def-C3}
C_3 = C_3(k) 
= a_3^2 C_1(k) \log (kH_0) ,
\\
&
\label{def-C4}
C_4  =C_4(k)
=C_1(k) a_1^2 \left( 1 + \frac1{\sqrt{C_3(k)}} \right)^2.
\end{align}
\end{lem}
\begin{proof}
We have from the definition of $F_X(\s,T)$ given as \eqref{def-FX} and Minkowski's inequality that 
\[
\sqrt{|F_X(\tfrac12,T)|}
\le \sqrt{\int_{0}^{T}  |\zeta(\tfrac{1}{2}+it) M_X(\tfrac{1}{2}+it)|^2   dt} + \sqrt{T}.
\]
To the last integral we apply Hiary's subconvexity bound \eqref{maxzeta} to bound zeta and 
\eqref{bound-moment-M-1/2} to bound the mean square of $M_X$. We let $I_0$ denote the resulting bound
so that
\[
I_0 = (a_1 T^{\frac16}\log T +a_2 )^2   (C_1T+C_2X) (\log X),
\]
and thus
\[
|F_X(\tfrac12,T)|
\le \left(\sqrt{I_0}+\sqrt{T}\right)^2 =I_0 \left( 1 + \sqrt{\frac{T}{I_0}} \right)^2.
\]
We note that $a_1T^{\frac16}\log T+a_2$ is minimized at $T=e^{-6}$ and we let $a_3$ represent this minimum. Then 
\[I_0\ge a_3^2(C_1T+C_2X)\log X\ge a_3^2C_1T\log X.\]
We conclude with the lower bound
$\frac{I_0}T\ge a_3^2C_1\log (kH_0)$, which is labeled $C_3$,
and
\[
I_0 
= C_1a_1^2  \left(T^{\frac16}\log T+\frac{a_2}{a_1}\right)^2  \left(T+\frac{C_2}{C_1}X\right) (\log X)     ,
\]
which completes the proof.
\end{proof}
\subsubsection{Bounding $F_X(\s_2,T)$ at $\s_2=1+\frac{\delta}{\log X}$}
\begin{lem}\label{bndfq1plusdelta}
 Let $T >0$, $X \ge kH_0$ and $k, \delta, \s_2$  satisfy \eqref{def-parameters}. Then
 \begin{equation}\label{bound-momentF-sigma2}
  F_X(\s_2,T) \le 
 \left( C_5(k,{\del} ) +  \frac{C_6(k,{\del} )(T+\pi m_0)}X  \right) (\log X)^2,
 \end{equation}
 where
 \begin{align}
& \label{def-C5}
C_5(k,{\del} ) = \frac{\pi m_0 b_{4} }{2{\del} } \Big(1+\frac{2{\del} }{\log (kH_0)} \Big)^2e^{\frac{2{\del} \gamma}{\log (kH_0)}}  ,
\\& 
\label{def-C6}
C_6(k,{\del} ) = \frac{b_{4} }{5{\del}  e^{\del} } \Big(1+\frac{{\del} }{\log (kH_0)} \Big)^2  +  \frac{b_{3} e^{-2{\del} }}{(\log (kH_0) )^2} ,
\end{align}
 the $b_i$'s are defined in \eqref{def-bi}, $m_0$ in \eqref{def-m0}, and $\gamma$ is Euler's constant.
\end{lem}
\begin{proof}
Recall that $F_X$ is defined by  \eqref{def-FX} and by 
 \eqref{series-fX}  we have
 \[ 
 F_X(\s_2,T) = \int_0^T|f_X(\s_2+it)|^2dt = \int_0^T \Big| \sum_{n \ge 1} \frac{\lambda_X(n)}{ n^{\s_2+it}} \Big|^2 dt .
 \]
Inequality \eqref{bound-MV-0T} implies the bound 
 \[
  F_X(\s_2,T)
\le \pi m_0 \sum_{n \ge 1} \frac{\lambda_X(n)^2}{n^{2\s_2-1}} +(T+\pi m_0)\sum_{n \ge 1}\frac{\lambda_X(n)^2}{n^{2\s_2}}.
\]
For $2\s_2-1 = 1+\frac{2{\del} }{\log X}$ and $2\s_2 = 2+\frac{2{\del} }{\log X}$, 
we apply the bounds for arithmetic sums \eqref{lambdaXntau1} 
and \eqref{lambdaXntau}
to respectively bound the two above sums. Thus
\begin{align*}
&  \sum_{n \ge1} \frac{\lambda_X(n)^2}{n^{1+2\frac{{\del} }{\log X}}} \le 
  \frac{b_{4} }{2{\del} } \Big(1+\frac{2{\del} }{\log X} \Big)^2e^{\frac{2{\del} \gamma}{\log X}} (\log X)^2 ,
\\
\text{and }\ 
&   \sum_{n \ge1} \frac{\lambda_X(n)^2}{n^{2+\frac{2{\del} }{\log X}}} 
\le 
\frac{b_{4} }{5{\del}  e^{\del} } \Big(1+\frac{{\del} }{\log X} \Big)^2   \frac{(\log X)^2}{ X} + \frac{b_{3} e^{-2{\del} }}{X}.
\end{align*}
We combine these results and use the fact that  $X\ge kH_0$ to complete the proof. 
\end{proof}
From here we may derive a bound for $\mathcal{M}_{g,T} (X, \s) $.
\subsection{Explicit upper bounds for the mollifier $\mathcal{M}_{g,T} (X, \s) $}
The results in this section are proven for a general weight $g$ satisfying the conditions described in Section \ref{Ingham-method}.
In \cite[Theorem 7]{HIP}, Hardy et al. proved the following convexity estimate:
\begin{lem} \label{convexityinequalityLem} Let $\frac12  \le \s_1 < 1 < \s_2$, let $T > 0$, and $X>1$.  Then
\begin{equation}
 \label{convexityinequality}
    \mathcal{M}_{g,T} (X, \s)  \le \mathcal{M}_{g,T} (X,\s_1)^{\frac{\sigma_2-\sigma}{\sigma_2-\sigma_1}} \mathcal{M}_{g,T} (X,\s_2)^{\frac{\sigma-\sigma_1}{\sigma_2-\sigma_1}}. 
\end{equation}
\end{lem}
In order to obtain a bound for the mollifier $\mathcal{M}_{g,T} (X, \s) $ inside the strip $\frac12  \le \sigma \le 1+ \frac{{\del} }{\log X}$, we need explicit bounds at the extremities $\frac12$ and $1+\frac{{\del} }{\log X}$.
\begin{lem}
Let $T > 0$, $X>0, \sigma \ge \frac{1}{2}$, and let $g$ satisfy conditions \eqref{gupbd} and \eqref{gabseven}.
Then
\begin{equation}\label{Msig1}
 \mathcal{M}_{g,T} (X, \s) \le
4 \omega_1(\s,T,\a)^2 \a \b \int_{0}^{\infty}  x^{\beta-1} e^{-2\a x^{\beta}}    F_X(\s,xT) dx.
 \end{equation}
\end{lem}
\begin{proof}
By \eqref{gabseven} and $|g(\sigma+it)| =|g(\sigma-it)|$ for $t \in \mathbb{R}$ and by an application of 
 \eqref{gupbd} to the weight $g $ in the definition \eqref{def-Mgsigma} of $\mathcal{M}_{g,T} (X, \s) $, we have 
\begin{equation}\label{bound1-Mg}
\mathcal{M}_{g,T} (X, \s)  \le 2 \omega_1(\s,T,\a)^2  \int_{0}^{\infty} e^{-2\a\left(\frac{t}{T}\right)^{\beta}} |f_X(\s+it)|^2dt.
\end{equation}
Note that $ \int_{0}^{U} |f_X(\s+it)|^2dt  =F_X(\s,U)$ with $F_X(\s,0)=0$ and
$\displaystyle{\lim_{U\to \infty} \Big( F_X(\sigma,U)e^{-2\a\left(\frac{U}{T}\right)^{\beta}} \Big) =0}$.
Integrating by parts gives
\begin{align*}
\int_{0}^{\infty} e^{-2\a\left(\frac{t}{T}\right)^{\beta}} |f_X(\s+it)|^2dt
& =  2\a\b \int_{0}^{\infty}\Big( \frac{t}{T}\Big)^{\beta} e^{-2\a\left(\frac{t}{T}\right)^{\beta}} F_X(\s,t) \frac{dt}t
\\ & =  2\a\b \int_{0}^{\infty} x^{\beta} e^{-2\a x^{\beta}} F_X(\s,xT) \frac{dx}x,
\end{align*}
by the variable change $x=\frac{t}{T}$.
This combined with \eqref{bound1-Mg} yields the announced \eqref{Msig1}.
\end{proof}
%%%
\subsubsection{Bounding $\mathcal{M}_{g,T} (X,\frac12)$}
Let $\a,\beta, A >0$ and let $n$ be a non-negative integer.
We define 
\begin{equation}\label{def-integralI-original}
 I(A,n)=\int_0^{\infty}x^{A}e^{-2\a x^{\beta}} (\log x)^{n}dx.
\end{equation}
In our context, $I(A,n)$ is a constant depending on parameters $A$ and $n$
and is $\mathcal{O}(1)$ in comparison with $T$.  
The change of variable $y=2\a x^B$ %, x= y^{1/B} (2\a)^{-1/B}, dx =(2\a)^{-1/B} B^{-1} y^{-1+1/B} dy$ 
leads to the identity
\begin{equation}\label{def-integralI}
I(A,n)  = (2\a)^{-\frac{A+1}{\beta} } \beta^{-(n+1)} \sum_{j=0}^n \binom{n}{j} (-\log(2\a))^{j} \Gamma^{(n-j)}\left(\frac{A+1}{\beta}\right),
\end{equation}
where $\Gamma^{(j)}(z)$ denotes the $j$-th derivative of Euler's gamma function. 
We also define
\begin{equation}
\begin{split}
\label{def-JkT}
 \mathcal{J} (k,T) =
& I(\beta+\tfrac13,0)
+ \frac{C_2}{C_1} k  I(\beta-\tfrac23,0)
+ \frac{2I(\beta+\frac13,1) + 2\frac{C_2}{C_1} k I(\beta-\frac23,1)}{(\log T)}\\
&+\frac{I(\beta+\frac13,2) + \frac{C_2}{C_1} k I(\beta-\frac23,2)}{(\log T)^2}
+ \frac{2a_2 \left( 
 I(\beta+\frac16,0) + \frac{C_2k}{C_1}  I(\beta-\frac56,0) \right)}{a_1T^{\frac16}(\log T)}\\
&+ \frac{2a_2 \left(I(\beta+\frac16,1)
+\frac{C_2k}{C_1} I(\beta-\frac56,1)
 \right)}{a_1T^{\frac16}(\log T)^2 }
+\frac{a_2^2 \left(I(\beta,0)+\frac{C_2k}{C_1}I(\beta-1,0)\right) }{a_1^2T^{\frac13}(\log T)^2 },
\end{split}
\end{equation}
\begin{equation} \label{def-mathcalU}
  \mathcal{U}(\alpha, k,T)
 = 4  \a \b C_4 \omega_1(\tfrac{1}{2},T,\a)^2 \mathcal{J} (k,T),
\end{equation}
where $\omega_1$ and $ C_4$ are respectively defined in \eqref{def-omega1} and \eqref{def-C4}.
We remark that in the case of our weight $g$, we have $\beta=2$.  Thus in our calculations of $
 \mathcal{J} (k,T)$ we  specialize to $\beta=2$.
\begin{lem}\label{M1half}
Let $\a,\b>0$ and $g $ be a function satisfying \eqref{gupbd} and \eqref{gabseven}.
Let  $T \ge H_0$, $X=kT$, and $k$ satisfies \eqref{def-parameters}.  Then
\begin{equation*}
\mathcal{M}_{g,T} (X,\tfrac12)
\le   \mathcal{U}(\alpha, k,T)
(\log (kT)) (\log T)^2 T^{\frac43} .
\end{equation*}
\end{lem}
\begin{proof}
We combine the bound \eqref{Msig1} for $\mathcal{M}_{g,T}$ with the bound \eqref{bound-momentF-1/2} for $F_X(\frac12,xT)$:
\begin{multline*}
   \mathcal{M}_{g,T} (X,\tfrac12)
\le 
4  \a \b C_4 \omega_1(\tfrac{1}{2},T,\alpha)^2 (\log X) 
\left \{
T^{\frac43} \int_{0}^{\infty}  x^{\beta+\frac13} (\log (xT))^2 e^{-2\a x^{\beta}}  dx
\right. \\ \left. 
+
\frac{2a_2}{a_1}T^{\frac76} \int_{0}^{\infty}  x^{\beta+\frac16} \log (xT)e^{-2\a x^{\beta}}  dx 
+\frac{a_2^2}{a_1^2}T \int_{0}^{\infty}  x^{\beta}  e^{-2\a x^{\beta}}  dx 
\right. \\ \left. +
\frac{C_2}{C_1}X T^{\frac13} \int_{0}^{\infty}  x^{\beta-\frac23} (\log (xT))^2 e^{-2\a x^{\beta}}  dx
+\frac{2a_2}{a_1}\frac{C_2}{C_1} XT^{\frac16} \int_{0}^{\infty}  x^{\beta-\frac56} \log (xT) e^{-2\a x^{\beta}}  dx 
\right. \\ \left. +
\frac{a_2^2}{a_1^2}\frac{C_2}{C_1} X \int_{0}^{\infty}  x^{\beta-1}  e^{-2\a x^{\beta}}  dx 
\right \}.
\end{multline*}
We also use the fact that $(\log(xT))^2 = (\log x)^2+2(\log x)(\log T)+(\log T)^2$ and obtain
\begin{multline*}
\mathcal{M}_{g,T} (X,\tfrac12)
\le   4  \a \b C_4 \omega_1(\tfrac{1}{2},T,\alpha)^2 (\log X) 
\left \{ T^{\frac43} 
 \left( 
I(\beta+\tfrac13,2) + 2 (\log T) I(\beta+\tfrac13,1)  \right.  \right. \\ \left. \left. + (\log T)^2 I(\beta+\tfrac13,0) 
\right)
  +\frac{2a_2}{a_1}T^{\frac76}
  \left( I(\beta+\tfrac16,1) + (\log T)I(\beta+\tfrac16,0) \right)
  + \frac{a_2^2}{a_1^2}T I(\beta,0)
  \right. \\ \left. 
  + \frac{C_2}{C_1}X T^{\frac13} 
  \left( I(\beta-\tfrac23,2) + 2(\log T) I(\beta-\tfrac23,1) + (\log T)^2 I(\beta-\tfrac23,0)  \right)
  \right. \\ \left. 
  +\frac{2a_2}{a_1}\frac{C_2}{C_1}XT^{\frac16} 
 \left(   I(\beta-\tfrac56,1) + (\log T) I(\beta-\tfrac56 ,0)  \right)
 + \frac{a_2^2}{a_1^2}\frac{C_2}{C_1}XI(\beta-1,0)
\right \},
\end{multline*}
where $I$ is the integral defined in \eqref{def-integralI-original}.
At this point we choose $X=kT$ so as to optimize the above bound, and 
we factor out the main term $T^{\frac43}(\log T)^2$:
\begin{multline*}
\mathcal{M}_{g,T} (X,\tfrac12)
\le
4  \a \b C_4 \omega_1(\tfrac{1}{2},T,\alpha)^2 (\log (kT)) (\log T)^2  T^{\frac43}
\left \{
 I(\beta+\tfrac13,0)
+\frac{k C_2}{C_1} 
 I(\beta-\tfrac23,0)
\right. \\ \left. 
+ 2\frac{ I(\beta+\tfrac13,1) +\frac{ k C_2  }{C_1} I(\beta-\tfrac23,1)}{(\log T)}
+ \frac{I(\beta+\frac13,2)
+ \frac{k C_2}{C_1}  I(\beta-\frac23,2)}{(\log T)^2}
+ \frac{2a_2 }{a_1 }  \frac{ I(\beta+\frac16,0) + \frac{kC_2 }{ C_1} I(\beta-\frac56,0) }{(\log T) T^{\frac16}} 
\right. \\ \left. 
+  \frac{2a_2 }{a_1} \frac{ I(\beta+\frac16,1) 
+ \frac{ kC_2  }{  C_1} I(\beta-\frac56,1)}{(\log T)^2 T^{\frac16}}
+ \frac{a_2^2}{a_1^2}
\frac{ I(\beta,0) + \frac{kC_2}{C_1}I(\beta-1,0)}{ (\log T)^2T^{\frac13}} 
\right \}.
\end{multline*}
We recognize in the above term between brackets $\mathcal{J} (k,T)$ as introduced in \eqref{def-JkT}.
%{\color{blue}We conclude by noting that both $J$ and $\omega_1$ are bounded by their value at $T=H_0$.
%Seems unnecessary.}
\end{proof}
\subsubsection{Bounding $\mathcal{M}_{g,T} (X,\s_2)$ at $\s_2 =1+\frac{\delta}{\log X}$}
\begin{lem}\label{M1delta}
Let $g $ be as defined in Lemma \ref{UpboundPhicor}.
Let $T \ge H_0$, $X=kT$, and
 $k, \delta, \s_2$  satisfy \eqref{def-parameters}.
% $\delta >0$, $\s_2 = 1+\frac{\delta}{\log X}$, and $X=kT$. 
%$\s_2,T,X$ satisfy \eqref{def-parameters} with $X=kT$. 
Then
\[
\mathcal{M}_{g,T} (X,\s_2) \le \mathcal{V} (\a,k,{\del},T )(\log (kT))^2,
\]
where
\begin{align}
& \label{def-mathcalV}
\mathcal{V} (\a,k,{\del} ,T)
 = 8 \a \omega_1(\s_2,T,\alpha)^2   \mathcal{K} (k,{\del} ,T),
\\
&\label{def-KkdeltaT}
 \mathcal{K} (k,{\del} ,T) = 
\left(C_5(k,{\del} ) + \frac{C_6(k,{\del} )\pi m_0}{kT } \right)
  I(1,0)
+ \frac{C_6(k,{\del} ) }k
 I(2,0),
\end{align}
and $m_0,\omega_1,C_5,C_6,$ and $I$ are respectively defined in \eqref{def-m0}, \eqref{def-omega1}, \eqref{def-C5}, \eqref{def-C6}, and  \eqref{def-integralI-original}.
\end{lem}
\begin{proof}
We combine the bound \eqref{Msig1} for $\mathcal{M}_{g,T} $ 
with the bound \eqref{bound-momentF-sigma2} for $F_X(\s_2,xT)$ (since $X \ge kH_0$) to obtain
\begin{equation}
 \mathcal{M}_{g,T} (X,\s_2)
\le 4 \a \b \omega_1(\s_2,T,\alpha)^2   
\Bigg(
\int_{0}^{\infty} x^{\beta-1} e^{-2\alpha x^{\beta}} 
\Big( C_5(k,\delta) + \frac{C_6(k,\delta)(xT+\pi m_0)  }{X} \Big) (\log X)^2 dx
\Bigg).
\end{equation} 
Rearranging this and recalling the definition for $I$ in \eqref{def-integralI-original} we obtain
\begin{multline*}
\mathcal{M}_{g,T} (X,\s_2)
\le 4 \a \b \omega_1(\s_2,T,\alpha)^2  (\log X)^2 \left(
\left(C_5(k,{\del} ) + \frac{C_6(k,{\del} )\pi m_0}X \right)
  I(\beta-1,0)
\right. \\ \left.  + \frac{C_6(k,{\del} )T}X 
 I(\beta,0)
\right).
\end{multline*}
We conclude by noting that $X=kT$ and for our $g$, $\beta=2$. 
% {\color{blue} noting that $\omega_1(\s_2,T,\alpha), C_5(k,{\del} )$ and $C_6(k,{\del} )$ decrease as $T$ increases. I don't think we
% use this. }
\end{proof}
% %%%
\subsubsection{Conclusion}
Finally, we provide bounds for $\mathcal{M}_{g,T} $.% for $g_1,g $. 
\begin{lem}
Let $g $ be as defined in Lemma \ref{UpboundPhicor}.
Let $T \ge H_0$, $X=kT$, and
 $k$  satisfies \eqref{def-parameters}.
Assume $\frac{1}{2} \le \s \le 1+\frac{\delta}{\log X}$.
Then
\begin{equation}\label{Mgsigma}
\begin{split}
  \mathcal{M}_{g,T} (X, \s)  
%   & \le {\color{blue} e^{
%  \frac{4\delta (2 \sigma-1)(\frac{2}{3} \log T + \log \log T  ) }{\log(kT)+2 \delta}   } }
%  \\
%  & \le {\color{blue} e^{
%  \frac{4\delta (2 \sigma-1)}{\log(kT)+2 \delta} (\frac{2}{3} \log T + \log \log T  )  } }
%  \\
& \le 
e^{\frac{8}{3} \delta (2 \sigma-1) M(k,\sigma) + \frac{ 4 \delta (2 \sigma-1)\log \log H_0}{\log(kH_0)+2 \delta} } 
\mathcal{U}(\alpha, k,T)^{2(1-\s)+ \frac{2 \delta (2 \sigma-1)}{\log(kT)+2 \delta}} \times \\
& \mathcal{V} (\a,k,{\del},T )^{2\s-1- \frac{2 \delta (2 \sigma-1)}{\log(kT)+2 \delta}}  
(\log (kT))^{2\s}
(\log T)^{4(1-\s)}
T^{\frac83(1-\s)},
\end{split}
\end{equation}
where $\mathcal{U}$ and $\mathcal{V}$ are respectively defined in \eqref{def-mathcalU} and \eqref{def-mathcalV}
and 
\begin{equation}
   \label{Mkdelta}
     M(k,\delta)  = \max \Big(  \frac{\log H_0}{\log(kH_0) + 2 \delta} , 1 \Big).
\end{equation}
\end{lem}
\begin{proof} 
Let $\s_1=\frac{1}{2}$ andf $\s_2 = 1+\frac{\delta}{\log X}$ and $\s \in [\s_1,\s_2]$. 
We apply the convexity inequality \eqref{convexityinequality} with exponents
\begin{equation}
  \label{abdefn}
 a = \frac{\s_2-\s}{\s_2-\s_1} = \frac{1+\frac{\delta}{\log X}-\s }{( 1+\frac{\delta}{\log X}) - \frac{1}{2}} 
 \text{ and }
 b = 1-a = \frac{\s-\s_1}{\s_2-\s_1} = \frac{\s -\frac{1}{2}  }{( 1+\frac{\delta}{\log X}) - \frac{1}{2} }
\end{equation} 
in combination with Lemmas \ref{M1half}, Lemma \ref{M1delta} to obtain 
\begin{equation}
  \label{MgTUVbd}
  \mathcal{M}_{g,T} (X, \s)  \le
\mathcal{U}(\alpha, k,T)^a 
\mathcal{V} (\a,k,{\del},T )^b 
(\log (kT))^{a+2b}
(\log T)^{2a}
T^{\frac43a}.
\end{equation}
Next, from the definitions of \eqref{abdefn} it may be checked that 
\begin{equation}
  \label{equalities}
a = 2(1-\s)+ \frac{2 \delta (2 \sigma-1)}{\log X +2 \delta}, 
\text{ and }
\  b = 2\s-1- \frac{2 \delta (2 \sigma-1)}{\log X +2 \delta}.
\end{equation}
From these equalities it follows that $a+2b \le 2 \s$. 
Using \eqref{equalities} and the bound for $a+2b$ (since $\log(kT) \ge \log(kH_0) \ge \log(10^9) >1$),
we have  
\begin{equation}\label{Mgsigma2}
\begin{split}
  \mathcal{M}_{g,T} (X, \s)  
  & \le 
  e^{\frac{4}{3} \times \frac{2 \delta (2 \sigma-1) \log T}{\log(kT)+2 \delta} 
  + 2 \times \frac{ 2 \delta (2 \sigma-1)\log \log T}{\log(kT)+2 \delta} } 
\mathcal{U}(\alpha, k,T)^{2(1-\s)+ \frac{2 \delta (2 \sigma-1)}{\log(kT )+2 \delta}} \times \\
& 
\mathcal{V} (\a,k,{\del},T )^{2\s-1- \frac{2 \delta (2 \sigma-1)}{\log(kT )+2 \delta}} 
(\log (kT))^{2\s}
(\log T)^{4(1-\s)}
T^{\frac83(1-\s)}.
\end{split}
\end{equation}
Next we observe that the function $\frac{\log T}{\log(kT) +2\delta}$
decreases if $\log k + 2\delta < 0$ and increases if $\log k + 2 \delta >0$ 
and thus  
\begin{equation}
  \label{M}
   \frac{\log T}{\log(kT) + 2\delta} \le 
    M(k,\delta) := 
    \begin{cases}
   \frac{\log H_0}{\log(kH_0) + 2 \delta} & \text{ if } \log k + 2\delta < 0, \\
   1 & \text{ if }  \log k + 2\delta  \ge  0
   \end{cases}
%   = M(k,\delta)
\end{equation}
where $M(k,\delta)$ was defined in \eqref{Mkdelta}.
Furthermore, it may be checked by the conditions on $k$, that
$\frac{\log \log T}{\log(kT) + 2\delta}$ decreases as long as $0 < \delta < \frac{\log(H_0) (\log \log H_0 -1)}{2} $.
Using these observations in \eqref{Mgsigma2} we deduce \eqref{Mgsigma}. 
%\begin{equation}\label{Mgsigma}
%\begin{split}
%  \mathcal{M}_{g,T} (X, \s)  
%  & \le e^{\frac{8}{3} \delta (2 \sigma-1) M(k,\delta) + \frac{ 4 \delta (2 \sigma-1)\log \log H_0}{\log(kH_0)+2 \delta} } 
%\mathcal{U}(\alpha, k,T)^{2(1-\s)+ \frac{2 \delta (2 \sigma-1)}{\log(kT)+2 \delta}} \times \\
%& \mathcal{V} (\a,k,{\del},T )^{2\s-1- \frac{2 \delta (2 \sigma-1)}{\log(kT)+2 \delta}} 
%(\log (kT))^{2\s}
%(\log T)^{4(1-\s)}
%T^{\frac83(1-\s)},
%\end{split}
%\end{equation}
% 
\end{proof}
%%%%%%%%%%
\subsection{Bounding $F_X(\sigma,T)-F_X(\sigma,H)$} 
% %
\begin{lem}\label{F1TF1Hbnd}
Let $g $ be as defined in Lemma \ref{UpboundPhicor}.
Let $\sigma\in[\frac12 ,1]$ and $\a>0$. 
Let $T \ge H_0 \ge H >0$, $X=kT$, 
 $k$  satisfies \eqref{def-parameters}, and $0 < \delta < \frac{\log(H_0) (\log \log H_0 -1)}{2} = 26.36 \ldots$.
%Let $H,T,X$ satisfy \eqref{def-parameters}. 
Then 
\begin{equation}
\begin{split}
F_X(\sigma,T)-F_X(\sigma,H)
& \le \frac{e^{\frac{8}{3} \delta (2 \sigma-1) M(k,\delta) + \frac{ 4 \delta (2 \sigma-1)\log \log H_0}{\log(kH_0)+2 \delta}  }\mathcal{U}(\alpha, k,T)^{2(1-\s)+\frac{2 \delta (2 \sigma-1)}{\log(kT)+2 \delta}}\mathcal{V} (\a,k,{\del} ,T)^{2\s-1-\frac{2 \delta (2 \sigma-1)}{\log(kT)+2 \delta}}}{2(\omega_2(\sigma,T,\a))^2 }  \\
& \times (\log (kT))^{2\s} (\log T)^{4(1-\s)} T^{\frac83(1-\s)},
\end{split}
\end{equation}
where $\omega_2, \mathcal{U}, \mathcal{V}$ are respectively defined in  \eqref{def-omega2}, \eqref{def-mathcalU}, \eqref{def-mathcalV}.
\end{lem}
\begin{proof}
By the assumed lower bound on $g$, \eqref{glbd}, we have 
\[
   F_X(\sigma,T)-F_X(\sigma,H)= \int_{H}^{T} |f_X(\s+it)|^2 dt \le \frac{1}{(\omega_2(\sigma,T,\a))^2 } \int_{H}^{T} |g(\s+it)|^2 |f_X(\s+it)|^2 dt.
\]
Since $t \to |g(\s+it)f_X(\s+it)|$ is even, it follows that 
\[
  F_X(\sigma,T)-F_X(\sigma,H) \le \frac{\mathcal{M}_{g,T} (X, \s) }{2(\omega_2(\sigma,T,\a))^2 }
\]
and we conclude by inserting the  bound \eqref{Mgsigma} for $\mathcal{M}_{g,T}(X,\s)$.
\end{proof}
% %%
\subsection{Explicit upper bounds for $ \int_{\s'}^{\mu} \arg h_X(\tau+iT ) d\tau -  \int_{\s'}^{\mu} \arg h_X(\tau+iH ) d\tau$}
The following Proposition and Corollary are a variant of Titchmarsh \cite[Lemma, p. 213]{Tit}.
This proposition gives a bounds for $ \arg f(\sigma+iT)$ where $f$ is a holomorphic function. 
The argument we use here is due to Backlund \cite{Back} in the case that $f(s)=\zeta(s)$.  
The cases of Dirichlet $L$-functions and Dedekind zeta functions have been worked out by 
McCurley \cite{Mc} and by the first and third authors \cite{KadNg} respectively.
\begin{prop} \label{Backlund}
Let $\eta >0$.  
Let $f(s)$ be a holomorphic function, for $\Re(s) \ge -\eta$, real for real $s$.  
Assume there exist positive constants $M$ and $m$ such that 
\begin{align}
\label{ubd1} & |f(s)| \le M \text{ for  }  \Re(s) \ge 1+\eta,  \\
\label{lbd} & |\Re f(1+\eta +it)| \ge m >0 \text{ for all } t \in \mathbb{R}. 
\end{align}
Let $\s\in(0,1+\eta]$ and assume that $U$ is not the ordinate of a zero of $f(s)$.  
Then there exists an increasing sequence of natural numbers $\{ N_k \}_{k=1}^{\infty}$ such that % if $N = N_{k}$ for some $k \in \mathbb{N}$, we have 
\begin{equation}
   \left|  \arg f(\sigma+iU) \right|
   \le \frac{\pi}{ \log 2} \mathscr{L}_k   + \frac{\pi \log M}{2 \log 2}-\frac{\pi \log m}{\log 2} +\frac{\pi}{2} +o_k(1)
\end{equation}
where 
\begin{equation}
  \label{def-Lk} 
   \mathscr{L}_k  =
    \frac{1}{2 \pi N_k} \int_{\frac{\pi}{2}}^{\frac{3 \pi}{2}}
    \log \Big( 
    \frac{1}{2} 
     \sum_{j=0}^{1}   |f(1+\eta +(1+2\eta) e^{i \theta}+(-1)^j iU)|^{N_k}
         \Big) \, d\theta
\end{equation}
and $o_k(1)$ is a term that approaches $0$ as $k \to \infty$.  
\end{prop}
\begin{proof}[Proof of Proposition \ref{Backlund}]
Let $\eta >0$. % and set $1+\eta  =1+\eta$.
We define $\arg f(1+\eta)=0$, and $\arg f(s) = \arctan \frac{\Im f(s)}{\Re f(s)}$ for $\Re(s)=1+\eta$,
since, by \eqref{lbd}, $\Re(f(s))$ does not vanish on $\Re(s) =1+\eta$.  It follows that 
\begin{equation}
  \label{f1eta}
  |\arg f(1+\eta+iU)| < \frac{\pi}{2}. 
\end{equation}
Recall that $\arg f(\sigma+iU)$ is defined by continuous variation, moving along the line $\mathcal{C}$
 from $1+\eta+iU$ 
to $\sigma +iU$.  It follows that 
\begin{equation}\label{bnd-fsigmaU}
  |\arg f(\sigma +i U)| \le \left| \Delta_{\mathcal{C}} \arg f(s) \right|+ \frac{\pi}{2}. 
\end{equation}
We now bound the argument change on $\mathcal{C}$. 
Let $N \in \mathbb{N}$ and let
\begin{equation}
   \label{def-FN}
   F_N(w) =\frac{1}{2}( f(w+iU)^N + f(w-iU)^N).
\end{equation}
Since $f(s)$ is real when $s$ is real, the reflection principle gives 
$F_N(\sigma) = \Re \, f(\sigma+iU)^N$ for all $\sigma$ real. 
Suppose $F_N(\sigma)$ has $n$ real zeros in the interval
$[\sigma, 1+\eta]$.  These zeros partition the interval
into $n+1$ subintervals.  On each of these subintervals $\arg f(\sigma+iU)^N$
can change by at most $\pi$, since $\Re \, f(\sigma +iU)^N$ is
nonzero on the interior of each subinterval.  It follows that 
\begin{equation}
   \label{changeargC}
   |\Delta_{\mathcal{C}} \arg f(s)|=
   \frac{1}{N} | \Delta_{\mathcal{C}} \arg f(s)^N|
   \le \frac{(n+1) \pi}{N}. 
\end{equation}
We now provide an upper bound for $n$.  
Jensen's theorem asserts that 
\[
   \log|F_N(1+\eta )| + \int_{0}^{1+2\eta} \frac{n(u)du}{u} =  \frac{1}{2 \pi} \int_{-\frac{\pi}{2}}^{\frac{3 \pi}{2}}
    \log |F_N(1+\eta +(1+2\eta) e^{i \theta}| d \theta,
\]
where $n(u)$ denotes the number of zeros of $F_N(z)$ in the circle centered at $1+\eta $ of radius $u$. 
Observe that $n(u) \ge n$ for  $u \ge \frac{1}{2}+\eta$ and thus
\begin{equation}
   \label{jensen}
    n \log 2 \le \frac{1}{2 \pi} \int_{-\frac{\pi}{2}}^{\frac{3 \pi}{2}}
    \log |F_N(1+\eta +(1+2\eta) e^{i \theta}| d \theta
    - \log|F_N(1+\eta )|. 
\end{equation}
Trivially from \eqref{def-FN}, 
\[
|F_N(1+\eta +(1+2\eta) e^{i \theta})| \le  \frac{1}{2} 
    \sum_{j=0}^{1}   |f(1+\eta +(1+2\eta) e^{i \theta}+(-1)^j iU)|^N, 
\]
so for the left part of the contour in \eqref{jensen},
\begin{equation}
  \label{left}
 \int_{\frac{\pi}{2}}^{\frac{3 \pi}{2}}
    \log |F_N(1+\eta +(1+2\eta) e^{i \theta}| d \theta  %\\   
\le 
\int_{\frac{\pi}{2}}^{\frac{3 \pi}{2}}
    \log \Big( 
   \frac{1}{2} 
    \sum_{j=0}^{1}   |f(1+\eta +(1+2\eta) e^{i \theta}+(-1)^j iU)|^N
    \Big) \, d\theta.
\end{equation}
For the right part of the contour in \eqref{jensen}, we have $ -\frac{\pi}{2} \le \theta \le \frac{\pi}{2}$, so $\Re(1+\eta +(1+2\eta)e^{i\theta})\ge 1+\eta$. We  apply \eqref{ubd1} and obtain
\begin{equation}
\label{right}
   \frac{1}{2 \pi} \int_{-\frac{\pi}{2}}^{\frac{\pi}{2}}    \log |F_N(1+\eta +(1+2\eta) e^{i \theta}| d \theta 
    \le \frac{N}{2} \log M.
\end{equation}
To complete our bound for $n$, we require a lower bound for $\log|F_{N_k}(1+\eta )|$.\\
We write $\displaystyle{f(1+\eta +iU)=re^{i \phi} }$ and then choose (by Dirichlet's approximation theorem) an increasing  sequence of positive integers $N_{k}$  tending to infinity such that $N_{k} \phi$ tends to $0$ modulo 
$2 \pi$. 
Since $\displaystyle{ \frac{F_{N_k}(1+\eta)}{|f(1+\eta+iU)|^{N_k}} =  \frac{r^{N_k} \cos(N_k\phi)}{r^{N_k}} }$, it follows that 
$\displaystyle{   \lim_{ k \to \infty} \frac{F_{N_k}(1+\eta)}{|f(1+\eta+iU)|^{N_k}}   = 1 }$. 
Thus we derive  
\[
\log |F_{N_k}(1+\eta)| \ge N_k \log |f(1+\eta+iU)| +o_k(1) ,
\]
where the term $o_k(1)\rightarrow 0$ as $k \rightarrow \infty$. 
Together with \eqref{lbd}, we obtain
\begin{equation}
  \label{lblogF}
\log |F_{N_k}(1+\eta)| \ge N_k \log m +o_k(1).
\end{equation}
Then \eqref{jensen}, \eqref{left}, \eqref{right}, and \eqref{lblogF} give
\begin{multline}
\label{nlog2B}
  n  \log 2  
\le \frac{1}{2 \pi}   \int_{\frac{\pi}{2}}^{\frac{3 \pi}{2}} \log \Big( \frac{1}{2} \sum_{j=0}^{1}   |f(1+\eta +(1+2\eta) e^{i \theta}+(-1)^j iU)|^{N_k} \Big) \, d\theta 
\\+ \frac{N_k  \log M }{2} - N_k\log m +o_k(1). 
\end{multline} 
By \eqref{changeargC} it follows that 
\[
  \left| \Delta_{\mathcal{C}} \arg f(s) \right|
   \le \frac{\pi}{\log 2} \mathscr{L}_k   + \frac{ \pi \log M}{2 \log 2}-\frac{ \pi \log m}{\log 2} +o_k(1) ,
\]
where $\mathscr{L}_k $ is defined by \eqref{def-Lk} . 
We conclude by combining this with \eqref{bnd-fsigmaU}. 
\end{proof}
We derive the following Corollary for $\arg h_X(s)$ from Proposition \ref{Backlund}.  
\begin{cor}  \label{arghX}
Let $\eta_0=0.23622\ldots$, $\eta \in [\eta_0, \frac{1}{2})$, and 
$X \ge 10^9$. Assume that $U \ge H \ge 1002$ and that $U$ is not the ordinate of a zero of $h_X(s)$. %=1-f_X(s)^2$. 
Then for all $\tau\in(0,1+\eta]$,
\begin{multline*}
   |\arg h_X(\tau+iU)|  
 \le \frac{(1+2\eta)}{\log 2} \log  \Big( \frac{b_{8}(\eta,H)}{2\pi}U \Big) 
 + \frac{\pi (1+ \eta) }{\log 2} ( \log X) +
  \frac{\pi \log  b_7(k,\eta,H_0) }{2\log 2}  
 + \frac{\pi \log b_{5}(\eta)}{2\log2} \\ - \frac{\pi \log(1-b_{6}(10^9,\eta)^2)}{\log 2} +\frac{\pi}2,
\end{multline*}
where $b_{5},b_{6},b_{7},b_{8}$ are defined in \eqref{def-b5}, \eqref{bound-logh-def-b6}, \eqref{def-b7}, and \eqref{def-b8}.
\end{cor}
%{\color{blue} NOTE to Allysa.  Definition of $b_6$ has changed to correct error.}
%%%
\begin{proof}[Proof of Corollary]
We apply Proposition \ref{Backlund} to $f=h_X$ as defined in \eqref{def-hXs}:
\[ h_{X}(s) = 1-f_{X}(s)^2 = \zeta(s) M_{X}(s) (2-\zeta(s)M_{X}(s)) . \]
Let $\s \ge \eta+1$ and $t\in \R$. We establish an upper bound for $|h_X(\s+it)|$.
The triangle inequality in conjunction with %\eqref{zetabd1} and \eqref{mollifierbd1} implies 
$\displaystyle{ |\zeta(s)| \le \zeta(1+\eta) }$ %\text{ for } \Re(s) \ge 1+\eta %\end{equation}
and with
$ \displaystyle{  |M_X(s)| \le \sum_{n=1}^{\infty} \frac{|\mu(n)|}{n^{1+\eta}} = \frac{\zeta(1+\eta)}{\zeta(2+2 \eta)} }$
give 
\begin{equation}\label{upper-hX}
|h_X(s)| \le b_{5}(\eta)
\end{equation}
with
\begin{equation} \label{def-b5}
b_{5}(\eta) =\frac{\zeta(1+\eta)^4}{\zeta(2+2 \eta)^2} +  \frac{2 \zeta(1+\eta)^2}{\zeta(2+2 \eta)}.
\end{equation}
We now give a lower bound for $|\Re h_X(1+\eta+it)|$.
We use the reverse triangle inequality $|h_X(s)| \ge 1 - |f_X(s)|^2$.
It remains to provide an upper bound for $|f_X(s)|$.  
Trivially from \eqref{def-fXs}, 
\[
  |f_X(s)|  \le
   \sum_{n > X}  \frac{|\lambda_{X}(n)|}{n^{1+\eta}}
  \le \sum_{n > X}  \frac{d(n)}{n^{1+\eta}} ,
\]
and by Lemma \ref{divisorsum}, we obtain  
\begin{equation}\label{bound-logh-def-b6}
|f_X(1+\eta+it)|  \le b_{6}(X,\eta) = \frac{(1+\eta) (\log X)}{ \eta X^{\eta}}  
   \Big(  1 +\frac{1}{\eta \log X} +\frac{\gamma }{ \log X}    +\frac{7 \eta}{12(1+\eta) X (\log X)} \Big).
\end{equation}
Note that $  \frac{  (\log X)}{ X^{\eta}}  $ decreases when $\eta > \frac{1}{\log X}$, which is the case since we assumed $\eta > \frac{1}{\log (10^9)}=0.048254\ldots$ and $X\ge 10^9$.
Thus  $|f_X(s)|   \le b_{6}(10^9,\eta) $
and
\begin{equation}\label{lower-hX}  
|\Re (h_X(s)) | = |1-\Re(f_X(s))^2| \ge |1-|f_X(s)|^2| \ge 1-|f_X(s)|^2 \ge 1-b_{6}(10^9,\eta)^2 .
\end{equation}
Note our assumption $\eta \ge \eta_0=0.23622\ldots$ ensures  $1-b_{6}(10^9,\eta)^2 >0$. \\
Finally, we must bound $\mathscr{L}_k $ as defined in \eqref{def-Lk} in the case $f=h_X$.  
We assume $w$ is a complex number such that $-\eta \le \Re w \le  1+\eta$ and $|\Im w| \ge U-(1+2\eta)$.  
Recall that by Lemma \ref{bndLschi} 
\[
    |\zeta(w)| \le  3 \frac{|1+w|}{|1-w|} \Big( \frac{|w+1|}{2 \pi} \Big)^{\frac{1+\eta-\Re w}{2}} \zeta(1+\eta) .
\]
Since
$
\frac{|1+w|}{|1-w|} = \Big| 1 + \frac{2}{w-1} \Big| 
\le 1 + \frac{2}{|\Im(w)|} \le  1.002$ when $|\Im(w)| \ge 1000$,
then 
\begin{equation}
   \label{zetabd}
    |\zeta(w)| \le 3.006
     \zeta(1+\eta) \Big( \frac{|w+1|}{2 \pi} \Big)^{\frac{1+\eta-u }{2}}  \text{ for } |\Im(w)| \ge 1000.
\end{equation}
From the definition \eqref{def-Mxs}, we have the trivial bound
\begin{equation}\label{MXbd}
 |M_X(w)|  \le X^{1+\eta}.
\end{equation}
It follows from 
\[|h_X(w)|  \le |\zeta(w)M_X(w)|^2 + 2 |\zeta(w)| |M_X(w)|
 ,\]
the bounds \eqref{zetabd}, \eqref{MXbd}, $ \frac{|w+1|}{2 \pi}>1, -\frac{1+\eta-\Re w }{2} <0$, and $X \ge kH_0$, 
that
\begin{equation}\label{bnd-hX}
 |h_X(w)| \le b_7(k,\eta,H_0) \Big( \frac{|w+1|}{2 \pi} \Big)^{1+\eta-u } X^{2(1+ \eta)}
  \text{ for } |\Im(w)| \ge 1000,
\end{equation}
with
\begin{equation}\label{def-b7}
b_7(k,\eta,H_0) = \left( 1 + \frac{2}{3.006 \zeta(1+\eta)(kH_0)^{1+\eta}}  \right)   \left(3.006 \zeta(1+\eta)\right)^2. 
\end{equation}
We apply this with $w=1+\eta+(1+2\eta)e^{i\theta}\pm iU$.
Since $\cos \theta\le 0$, a little calculation gives 
\[|w+1|  = |2+\eta+(1+2\eta)e^{i\theta}\pm iU|
\le  \sqrt{ (2+\eta)^2+ (1+2\eta+U)^2} 
\le b_{8}(\eta,H) U ,\]
with
\begin{equation}\label{def-b8}
b_{8}(\eta,H)=\sqrt{ \frac{(2+\eta)^2}{H^2}+ \Big(\frac{1+2\eta}H+1\Big)^2}.
\end{equation}
In addition $1+\eta-u = 1+\eta- (1+\eta+(1+2\eta)\cos \theta) =  -(1+2\eta)(\cos \theta) $,
and \eqref{bnd-hX} gives
\begin{equation}\label{bnd-hX-eta}
 |h_X(1+\eta+(1+2\eta)e^{i\theta}\pm iU)|
  \le  b_7(k,\eta,H_0)   \Big( \frac{b_{8}(\eta,H)}{2\pi}U \Big)^{-(1+2\eta)(\cos \theta)} X^{2(1+ \eta)} ,
\end{equation}
since $ | \Im(1+\eta+(1+2\eta)e^{i\theta}\pm iU) | \ge U-(1-2\eta) \ge H-2 \ge 1000$. 
We use this to bound $\mathscr{L}_k $ as defined in \eqref{def-Lk}: 
\[
\mathscr{L}_k  
\le \frac{1}{2\pi } \int_{\frac{\pi}{2}}^{\frac{3 \pi}{2}} \left(
\log  b_7(k,\eta,H_0)  
- (1+2\eta)(\cos \theta) \log  \Big( \frac{b_{8}(\eta,H)}{2\pi}U \Big) 
+ 2(1+ \eta) ( \log X) \right)
     d\theta .
\]
Calculating the integrals give
\begin{equation}\label{bnd-Lketa}
\mathscr{L}_k  
\le \frac{\log  b_7(k,\eta,H_0)  }{2 } 
+\frac{(1+2\eta)}{\pi} \log  \Big( \frac{b_{8}(\eta,H)}{2\pi}U \Big) 
+ (1+ \eta) ( \log X) .
\end{equation}
%We put together this bound for $\mathscr{L}_k $, \eqref{upper-hX} with $M=b_{5}(\eta)$, and \eqref{lower-hX} with $m=1-b_{6}(10^9,\eta)^2$. 
By \eqref{upper-hX} and \eqref{lower-hX} we may take $M=b_5(\eta)$ and $m=1-b_6(10^9,\eta)^2$ in 
\eqref{ubd1} and \eqref{lbd} in the case of $f(s)=h_X(s)$.   Therefore by Proposition  \ref{Backlund}
 \begin{equation}
   \left|  \arg h_X(\sigma+iU) \right|
   \le \frac{\pi}{ \log 2} \mathscr{L}_k   + \frac{\pi \log b_5(\eta)}{2 \log 2}-\frac{\pi \log(1-b_6(10^9,\eta)^2)}{\log 2} +\frac{\pi}{2} +o_k(1).
\end{equation}
Inserting the upper bound for  $\mathscr{L}_k$ from \eqref{bnd-Lketa} and letting $k \to \infty$ we complete the proof
as the $o_k(1)$ terms goes to zero.
\end{proof}
We are now in a position to bound the arguments. 
\begin{lem}\label{finalargbndzeta}
Let $0 < H \le H_0 \le  T$ and $X \le T$. 
%Let $H,T,X$ satisfy \eqref{def-parameters}.
Let %$12\le H<H_0\le T$, %$k\ge 10^9 $, and 
$\eta \in (\eta_0, \frac{1}{2})$ with $\eta_0=0.23622\ldots$, 
$\sigma'$ and $\mu$ satisfying $\frac12 \le \sigma'<1 <\mu \le 1+\eta$.
Then
\begin{equation}
  \label{integralargs0}
\Big| \int_{\s'}^{\mu} \arg h_X(\tau+iT ) d\tau -  \int_{\s'}^{\mu} \arg h_X(\tau+iH )d\tau  \Big| \le C_7(\eta,H)  \, (\mu-\sigma')  (\log T) ,
\end{equation}
where 
 \begin{equation}\label{def-C7}
C_7(\eta,H) = \frac{2(1+2\eta)+2\pi (1+ \eta) }{\log 2} + \frac{b_{9}(\eta,H)  }{\log H_0}.
\end{equation} 
with $b_{9}(\eta,H)$ defined in \eqref{def-b9}.
%with $C_7(\eta,H)$ defined in \eqref{def-C7}. 
%, $\eta_1 = 0.256192\ldots $, and \begin{equation} \label{def-C7-eta0}C_7(\eta_1)= 17.585633 \ldots.\end{equation}
\end{lem}
% %
%
\begin{proof}
Note that 
\begin{equation}
  \label{integralargs}
 \Big| \int_{\s'}^{\mu} \arg h_X(\tau+iT ) d\tau -  \int_{\s'}^{\mu} \arg h_X(\tau+iH )d\tau \Big|
 \le (\mu-\s') \max_{\tau \in (\s',\mu)} \Big( | \arg h_X(\tau+iT )| + |\arg h_X(\tau+iH )| \Big). 
\end{equation}
By Corollary \ref{arghX} we have
\begin{equation}
  \label{sumarguments}
  |\arg h_X(\tau +iH)  | +   |\arg h_X(\tau+iT)  |
 \\
 \le b_{9}(\eta,H) 
+\frac{(1+2\eta)}{\log 2} (\log (HT)  )
+ \frac{2\pi (1+ \eta) }{\log 2} ( \log X) 
\end{equation}
with
\begin{equation}\label{def-b9}
b_{9}(\eta,H) =  \frac{\pi \log  b_7(k,\eta,H_0) }{\log 2}  
 + \frac{\pi \log b_{5}(\eta)}{\log2} 
 - \frac{2\pi \log(1-b_{6}(10^9,\eta)^2)}{\log 2} 
 + \pi 
+ \frac{2(1+2\eta)}{\log 2} \log  \Big( \frac{b_{8}(\eta,H)}{2\pi} \Big) 
\end{equation}
where $b_7,b_5,b_6,b_8$ are defined in \eqref{def-b7}, \eqref{def-b5}, \eqref{bound-logh-def-b6}, \eqref{def-b8}. 
Factoring $\log T$ in the right hand side of  \eqref{sumarguments}, using $H \le T$, $X \le T$, and $H_0 \le T$
yields
\begin{equation}
  \label{sumarguments2}
  |\arg h_X(\tau +iH)  | +   |\arg h_X(\tau+iT)  |
 \le 
(\log T)
\left( 
\frac{2(1+2\eta)+2\pi (1+ \eta) }{\log 2} 
+ \frac{b_{9}(\eta,H)  }{\log H_0}
\right).
\end{equation}
Combining \eqref{integralargs} and \eqref{sumarguments2} leads to \eqref{integralargs0}.
%We define
%\begin{equation}\label{def-C7}
%C_7(\eta,H) = \frac{2(1+2\eta)+2\pi (1+ \eta) }{\log 2} + \frac{b_{9}(\eta,H)  }{\log H_0} .
%\end{equation} 
\end{proof}
\subsection{Explicit lower bounds for $\int_{H}^{T} \log | h_X(\mu+it)| dt$}
%%%
First, observe that  \eqref{bound-logh-def-b6}  implies for 
\begin{equation}
  \mu \ge 1+\eta_0 =1.23622\ldots, \ |f_X(\mu+it)|<1.
\end{equation}
This fact is used in the next lemma. 
\begin{lem}\label{lowerbndlog}
Assume $\mu \ge 1+\eta_0$ where $\eta_0=0.23622 \ldots$. 
Let $X =kT$ where $T \ge H_0$, $k$ satisfies  \eqref{def-parameters}, $k \le 1$, and $2\pi m_0 \le  H < T$. 
%Let $H,T,X$ satisfy \eqref{def-parameters}. 
Then
 \begin{equation}
 -\int_{H}^{T} \log | h_X(\mu+it)| dt \le C_8(k,\mu) (\log T).
 \end{equation}
with
\begin{equation}\label{def-C8}
 C_8(k,\mu) =  b_{10}(k,\mu)  \frac{(\log (kH_0))^2}{(kH_0)^{2\mu-2}} \left( \frac{4\mu b_{11}(kH_0,2\mu) }{k(2\mu-1) }  + \frac{2\pi m_0 (2\mu-1) b_{11}(kH_0,2\mu-1)}{(\mu-1)}  \right),
\end{equation}
$b_{10}$ is defined in \eqref{def-b10}, $b_{11}$ in \eqref{def-b11}, and $m_0$ in \eqref{def-m0}.
\end{lem}
\begin{proof}
We begin by remarking that \eqref{bound-logh-def-b6} implies $|f_X(\mu+it)| \le b_{6}(kH_0,\mu-1)<1$
since $X \ge kH_0 \ge 10^9$ and  $\mu \ge 1+\eta_0$.
Next, observe that $
   |h_X(\mu+it)| \ge |1 - f_X(\mu+it)^2| \ge 1 - |f_X(\mu+it)|^2$
and thus   
\begin{equation}
  \label{loghX}
 -\log|h_X(\mu+it)| \le -\log(1-|f_X(\mu+it)|^2).
\end{equation}
Since $ - \frac{\log(1-u^2)}{u^2} $ increases with $u\in(0,1)$, we have 
\begin{equation}
  \label{logfX}
-\log(1-|f_X(\mu+it)|^2)\le b_{10}(k,\mu) |f_X(\mu+it)|^2 ,
\end{equation} 
with
\begin{equation}\label{def-b10}
b_{10}(k,\mu) = - \frac{\log\left(1-b_{6}(kH_0,\mu-1)^2\right)}{b_{6}(kH_0,\mu-1)^2} 
\end{equation}
where $b_6$ is defined in \eqref{bound-logh-def-b6}.
%Note that  $
%   |h_X(\mu+it)| \ge |1 - f_X(\mu+it)^2| \ge 1 - |f_X(\mu+it)|^2$
%and thus   
%\begin{equation}
% -\log|h_X(\mu+it)| \le -\log(1-|f_X(\mu+it)|^2).
%\end{equation}
It follows from \eqref{loghX} and \eqref{logfX} that 
\begin{equation}
   \label{bdinthxmu}
   -\int_{H}^{T} \log | h_X(\mu+it)| dt 
 \le b_{10}(k,\mu) \int_{H}^{T} |f_X(\mu+it)|^2 dt. 
\end{equation}
We apply Lemma \ref{ExplicitMV} and the bound $|\lambda_X(n)| \le d(n)$ with $\lambda_{X}(n)=0$ if $n \le X$. 
We obtain
\begin{align*}
    \int_{H}^{T} |f_X(\mu+it)|^2 dt & \le \sum_{n=1}^{\infty} \frac{|\lambda_X(n)|^2}{n^{2\mu}} (T-H + 2\pi m_0(n+1)) \\
    & \le (T-H +2\pi m_0) \sum_{n > X} \frac{d(n)^2}{n^{2\mu}}
    + 2\pi m_0 \sum_{n > X} \frac{d(n)^2}{n^{2\mu-1}}.
\end{align*}
We appeal to \eqref{divisorsum2} to bound the above sums:
\[
\sum_{n\ge X}\frac{d(n)^2}{n^\tau}
\le \frac{ (\log X)^3}{ X^{\tau-1}} \frac{2\tau b_{11}(kH_0,\tau) }{(\tau-1) }  ,
\]
since $X \ge kH_0$  where
\begin{equation}\label{def-b11}
 b_{11}(X,\tau)=  1+ \frac{3 }{(\tau-1)(\log X)}+ \frac{6}{(\tau-1)^2(\log X)^2} + \frac{6}{(\tau-1)^3(\log X)^3} .
\end{equation}
Since $X=kT$ we deduce that 
\[
    \int_{H}^{T} |f_X(\mu+it)|^2 dt
 \le \frac{(\log (kT))^3}{(kT)^{2\mu-2}}
\left( \frac{4\mu b_{11}(kH_0,2\mu) }{k(2\mu-1) } 
    + \frac{2\pi m_0 (2\mu-1) b_{11}(kH_0,2\mu-1)}{(\mu-1)}  \right).
\]
Note that $\frac{(\log (kT))^2}{(kT)^{2\mu-2}}$ decreases with $T$ as long as $ 10^9 > e^{\frac1{\mu-1}}$ (i.e. $\mu>\mu_2=1.072382\ldots$). 
Using this and  $\log(kT) \le \log T$ (since $k \le 1$) implies
\begin{equation}
    \label{fxmu2bd}
    \int_{H}^{T} |f_X(\mu+it)|^2 dt 
\le \frac{(\log (kH_0))^2}{(kH_0)^{2\mu-2}} \left( \frac{4\mu b_{11}(kH_0,2\mu) }{k(2\mu-1) }  + \frac{2\pi m_0 (2\mu-1) b_{11}(kH_0,2\mu-1)}{(\mu-1)}  \right) (\log T).
 \end{equation}
We conclude by combining this with \eqref{bdinthxmu}.
\end{proof}
%%%
\subsection{Proof of Zero Density Result }
Finally, we are able to compile our bounds to obtain an upper bound for $N(\sigma, T)$. 
\begin{lem}\label{defNsigTlem}
Assume $\alpha>0,d>0, \delta >0,  \eta_0=0.23622 \ldots, \eta \in [\eta_0,\frac{1}{2}),$ and $\mu \in [1+\eta_0,1+\eta]$.
Let $H_0=3.0610046 \cdot 10^{10}, \ 
1002\le H\le  H_0 , \ \frac{10^9}{H_0} \le  k \le 1$, $T  \ge H_0$, and $X = kT$ 
%$X \ge kH_0\ge 10^9$. 
%, and  $\delta>0$. %,\ \sigma_1 = \frac12,\ \text{and}\ \s_2=1+\frac{{\del} }{\log X}$.
Assume $\s>\frac12+\frac{d}{ \log H_0 }$, $ \mathcal{U}(\alpha, k,H_0) >1$,
and $\mathcal{U}(\alpha,k,T)$ decreases in $T$.
Thus
\begin{equation}\label{bnd-N@H0}
N(\s,T)
\le \frac{(T-H)(\log T)}{2\pi d}
\log
\left( 1 
+ \mathcal{C}_1 \frac{ 
(\log (kT))^{2\s}  (\log T)^{4(1-\s)} T^{\frac83(1-\s)}
 }{T-H}  \right)
+ \frac{ \mathcal{C}_2 }{2\pi d} (\log T)^2 ,
\end{equation}
\begin{equation}\label{bnd-Nasymp}
N(\s,T)
\le \frac{\mathcal{C}_1  }{2\pi d}
(\log (kT))^{2\s}  (\log T)^{5-4\s} T^{\frac83(1-\s)}
+ \frac{ \mathcal{C}_2 }{2\pi d} (\log T)^2 ,
\end{equation}
with
\begin{align}
 \label{def-mathcalC1}
\mathcal{C}_1= \mathcal{C}_1(\alpha, d,\del, k, H, \s)  & = 
 b_{12}(H)  e^{\frac{8}{3} \delta (2 \sigma-1) M(k,\delta) + \frac{ 4 \delta (2 \sigma-1)\log \log H_0}{\log(kH_0)+2 \delta} }     \mathcal{U}(\alpha, k,H_0)^{2(1-\s) +\frac{2d}{\log H_0 }+ \frac{2 \delta (2 \sigma-1)}{\log(kH_0)+2 \delta}} \times \\
 & \nonumber
  \mathcal{V}(\alpha, k,{\del} ,H_0)^{2\s-1}  e^{\frac{2d (2 \log\log H_0 -\log \log (kH_0))}{\log H_0 }+\frac{8d}{3 } +2\alpha},
\\  \label{def-mathcalC2} \mathcal{C}_2  =\mathcal{C}_2 (d, \eta, k, H, \mu,\s) & =  C_7(\eta,H) 
\Big(\mu- \s+\frac{d}{\log H_0 }\Big)   +   C_8(k,\mu) ,
\end{align}
and $\mathcal{U}, \mathcal{V}, M(k,\delta), C_7, C_8$ and $b_{12}$ are respectively defined in \eqref{def-mathcalU}, \eqref{def-mathcalV}, \eqref{Mkdelta}, \eqref{def-C7}, \eqref{def-C8}, and \eqref{def-b12}.\\
\end{lem}
\noindent {\it Remark}. 1. The assumptions that $U(\alpha,k, H_0)>1$ and $U(\alpha,k,T)$ are decreasing can be 
removed from the theorem.  However, this would overly complicate the statement of the theorem.  
In all instances that we apply this theorem (for various values of $\alpha$ and $k$) these conditions hold.
%
%Note that, since $\beta=2$, $4C_4 %= 4C_1(k) a_1^2 \left( 1 + \frac1{\sqrt{C_3(k)}} \right)^2\simeq 4C_1(k) a_1^2
% \simeq  \frac{24 a_1^2}{\pi^2} $, $\frac{C_2}{C_1} = o(1)$, and $ I(A,0)  = (2\a)^{-\frac{A+1}{\beta} } \beta^{-1} \Gamma(\frac{A+1}\beta)$, then
%\begin{align*}
%\mathcal{U}(\alpha, k,T) 
%& \le (1+o(1))  4 \a \b C_4 \omega_1(\tfrac{1}{2},T,\a)^2 
%\left(  I(\beta+\tfrac13,0)
%+ \frac{C_2}{C_1} k  I(\beta-\tfrac23,0)
% \right)
%\\
%& \le (1+o(1)) \frac{12 \Gamma\left(\frac{5}{3}\right) (2\a)^{-\frac2{3}} \left( a_1 \omega_1(\tfrac{1}{2},T,\a)\right) ^2 }{\pi^2} .
%\end{align*}
%
\begin{proof}
We begin by assuming that $T$ is not the ordinate of a zero of $\zeta(s)$. 
From \eqref{boundN2}, \eqref{Jensenbnd}, and the definition \eqref{def-FX} of $F_X$, we have 
for $\sigma \in [\sigma',1]$ where $\sigma' \ge \frac{1}{2}$ and $\mu  \in [1+\eta_0,1+\eta]$
\begin{multline*}
N(\s,T)  
\le \frac{1}{2\pi(\s-\s')}   \Big( 
(T-H) \log\left( 1  +  \frac{F_X(\s',T)-F_X(\s',H)}{(T-H)}  \right)
\Big. \\ \Big.+    \int_{\s'}^{\mu} \arg h_X(\tau+iT) d\tau 
-  \int_{\s'}^{\mu} \arg h_X(\tau+iH) d\tau 
  - \int_{H}^{T} \log | h_X(\mu+it)| dt
\Big).
\end{multline*}
We apply Lemma \ref{F1TF1Hbnd}, Lemma \ref{finalargbndzeta}, and Lemma \ref{lowerbndlog} to achieve
\begin{equation}
\begin{split}
& N(\s,T)  
\le \frac{(T-H) }{2\pi(\s-\s')} \times  \\
& \log\Bigg( 1  +  \frac{e^{\frac{8}{3} \delta (2 \sigma'-1) M(k,\delta) + \frac{ 4 \delta (2 \sigma'-1)\log \log H_0}{\log(kH_0)+2 \delta} }\mathcal{U}(\alpha, k,T)^{2(1-\s')+ \frac{2 \delta (2 \sigma'-1)}{\log(kT)+2 \delta}}
\mathcal{V} (\a,k,{\del} ,T)^{2\s'-1- \frac{2 \delta (2 \sigma'-1)}{\log(kT)+2 \delta}}}{2(\omega_2(\sigma',T,\a))^2 } \times \\
& 
 \frac{ (\log (kT))^{2\s'} (\log T)^{4(1-\s')} T^{\frac83(1-\s')}}{(T-H)}  \Bigg)   +   
 \frac{\left(  C_7(\eta,H)  \, (\mu-\sigma')   +   C_8(k,\mu)\right)(\log T)}{2\pi(\s-\s')} .
\end{split}
\end{equation}
We make the choice $\s'=\s-\frac{d}{\log T }$, for some $d>0$. 
From the definition  \eqref{def-mathcalV}, we note that  $\mathcal{V} (\a,k,{\del},T )$  decreases with $T$. 
Since by assumption $ \mathcal{U}(\alpha, k,H_0) >1$ and $T \to U(k,\alpha,T)$ decreases, 
it follows that $\mathcal{U}(\alpha, k,H_0)^{\frac{2d}{\log T }+\frac{2 \delta (2 \sigma'-1)}{\log(kT)+2 \delta}
}$ decreases with $T$ % and $\mathcal{V} (\a,k,{\del},T )>1$
and thus
\[
 \mathcal{U}(\alpha, k,T)^{2(1-\s')+ \frac{2 \delta (2 \sigma'-1)}{\log(kT)+2 \delta}} 
\le 
\mathcal{U}(\alpha, k,H_0)^{2(1-\s) +\frac{2d}{\log H_0 }+ \frac{2 \delta (2 \sigma'-1)}{\log(kH_0)+2 \delta}}.
\]
It may be shown that for our choice of parameters $\alpha, k, \delta$ that 
$ \mathcal{V} (\a,k,{\del} ,T) > 1$ for all $T \ge H_0$ and thus 
\[
 \mathcal{V} (\a,k,{\del} ,T)^{2\s'-1- \frac{2 \delta (2 \sigma'-1)}{\log(kT)+2 \delta}}
 \le  \mathcal{V} (\a,k,{\del} ,T)^{2\s'-1}. 
\]
In addition,
\begin{align*}  
(\log (kT))^{2\s'} (\log T)^{4(1-\s')} T^{\frac83(1-\s') } 
& =  e^{\frac{2d}{\log T} (2 \log \log T - \log \log (kT))  + \frac{8d}{3} } (\log (kT))^{2\s}  (\log T)^{4(1-\s)} T^{\frac83(1-\s)}   \\
& \le  e^{\frac{2d(2 \log \log H_0 - \log \log (kH_0))}{\log H_0 }+\frac{8d}{3 } }
(\log (kT))^{2\s}  (\log T)^{4(1-\s)} T^{\frac83(1-\s)},
\end{align*}
since $T \ge H_0$ and $\frac{10^9}{H_0} \le k \le 1$ imply $\frac{2 \log \log T-\log \log (kT)}{\log T }$  decreases in $T$.
Since $\omega_2(\sigma',T,\a)$ as defined in \eqref{def-omega2} increases with $\s'\ge \s-\frac{d}{\log H_0}$ and decreases with $T$, then
\begin{equation}\label{def-b12}
\frac1{2(\omega_2(\sigma',T,\a))^2}  \le b_{12}(H)  e^{2\alpha} \ \text{with}\ b_{12}(H)= \frac1{2 (1-\frac1H)^2}. 
\end{equation}
Combining the above inequalities establishes \eqref{bnd-N@H0}, and thus \eqref{bnd-Nasymp} (applying $\log(1+y)\le y$).  
\begin{equation}
\begin{split}
N(\s,T)  
& \le \frac{(T-H) (\log T)}{2\pi d}
\log\left( 1  +  b_{12}(H)  
e^{\frac{8}{3} \delta (2 \sigma'-1) M(k,\delta) + \frac{ 4 \delta (2 \sigma'-1)\log \log H_0}{\log(kH_0)+2 \delta} 
}  \right.  \\
& \times \mathcal{U}(\alpha, k,H_0)^{2(1-\s) +\frac{2d}{\log H_0 }+ \frac{2 \delta (2 \sigma'-1)}{\log(kH_0)+2 \delta}}   
 \mathcal{V}(\alpha, k,{\del} ,H_0)^{2\s-1} e^{\frac{2d (2 \log \log H_0 - \log \log (kH_0))}{\log H_0 }+\frac{8d}{3 } +2\alpha} \\
 & \times
 \left.
\frac{  (\log (kT))^{2\s}  (\log T)^{4(1-\s)} T^{\frac83(1-\s)} }{(T-H)}  \right) \\
& + \frac{\left(  C_7(\eta,H)  \, (\mu- \s+\frac{d}{\log H_0 })   +   C_8(k,\mu)\right)(\log T)^2}{2\pi d} .
\end{split}
\end{equation}
Since $\sigma' \le \sigma$, each remaining occurrence of $\sigma'$ may be replaced by $\sigma$. 
Finally, by a continuity argument these inequalities extend to the case where $T$ is the ordinate of a zero of the zeta function. 
\end{proof}
%%%
\section{Tables of Computation}\label{tables}
For fixed values of $\s$, Table \ref{g2asymptotick} provides bounds for $N(\s,T)$ of the shape \eqref{bnd-Nasymp}.
We fix values for $k$ in $[\frac{10^9}{H_0},1]$.
The parameters $\alpha, d, \del, \eta$  and  $H$ are chosen to make $\frac{ \mathcal{C}_1}{2\pi d}$ as small as possible with $\mathcal{C}_1(\alpha, d, \del,k, H, \s)$ as defined in \eqref{def-mathcalC1} .
The program returns $H=H_0-1$ for all lines in the table. With this $H$  we minimize of $C_7(\eta,H)$ which chooses $\eta=0.25618 \ldots$.
Then $\mu$ is chosen to minimize $\mu C_7(\eta,H)+C_8(k,\mu)$ 
(as in the definition \eqref{def-mathcalC2} of  $\mathcal{C}_2 =\mathcal{C}_2 (d, \eta, k, H, \mu,\s)$). We remark that there is a small bit of subtlety when considering $\mathcal{U}(\alpha,k,T)$, it is necessary to ensure all the coefficients in $\mathcal{J}(k,T)$ are positive and this is checked with each set of parameters used. This is to guarantee that $\mathcal{U}(\alpha,k,T)$ decreases with $T$.  \\
\begin{small}
\begin{table}[ht]
 \centering
\caption{The bound $N(\s,T)
\le A
(\log (kT))^{2\s}  (\log T)^{5-4\s} T^{\frac83(1-\s)}
+ B (\log T)^2$  \eqref{bnd-Nasymp}  for $\sigma = \sigma_0$ with $\frac{10^9}{H_0}\le k \le 1$.}
\label{g2asymptotick}
\begin{tabular}{|c|c|c|c|c|c|c|c|}
\hline
$\s_0$ & $k$ & $\mu$ &$\alpha$ &  ${\del} $ & $d$ &  $A=\frac{\mathcal{C}_1}{2\pi d}$ & $B=\frac{\mathcal{C}_2}{2\pi d}$ \\
 \hline
$ 0.60 $&$ 0.5 $&$ 1.251 $&$ 0.288 $&$ 0.3140 $&$ 0.341$&$ 2.177 $&$ 5.663 $\\
$ 0.65 $&$ 0.6 $&$ 1.249 $&$ 0.256 $&$ 0.3070 $&$ 0.340$&$ 2.963 $&$ 5.249 $\\
$ 0.70 $&$ 0.8 $&$ 1.247 $&$ 0.222 $&$ 0.3040 $&$ 0.339$&$ 3.983 $&$ 4.824 $\\
$ 0.75 $&$ 1.0 $&$ 1.245 $&$ 0.189 $&$ 0.3030 $&$ 0.338$&$ 5.277 $&$ 4.403 $\\
$ 0.80 $&$ 1.0 $&$ 1.245 $&$ 0.160 $&$ 0.3030 $&$ 0.337$&$ 6.918 $&$ 3.997 $\\
$ 0.85 $&$ 1.0 $&$ 1.245 $&$ 0.133 $&$ 0.3030 $&$ 0.336$&$ 8.975 $&$ 3.588 $\\
$ 0.86 $&$ 1.0 $&$ 1.245 $&$ 0.127 $&$ 0.3030 $&$ 0.335$&$ 9.441 $&$ 3.514 $\\
$ 0.87 $&$ 1.0 $&$ 1.245 $&$ 0.122 $&$ 0.3030 $&$ 0.335$&$ 9.926 $&$ 3.430 $\\
$ 0.88 $&$ 1.0 $&$ 1.245 $&$ 0.116 $&$ 0.3030 $&$ 0.335$&$ 10.431 $&$ 3.346 $\\
$ 0.89 $&$ 1.0 $&$ 1.245 $&$ 0.111 $&$ 0.3030 $&$ 0.335$&$ 10.955 $&$ 3.262 $\\
$ 0.90 $&$ 1.0 $&$ 1.245 $&$ 0.105 $&$ 0.3030 $&$ 0.334$&$ 11.499 $&$ 3.186 $\\
$ 0.91 $&$ 1.0 $&$ 1.245 $&$ 0.100 $&$ 0.3030 $&$ 0.334$&$ 12.063 $&$ 3.102 $\\
$ 0.92 $&$ 1.0 $&$ 1.245 $&$ 0.095 $&$ 0.3030 $&$ 0.334$&$ 12.646 $&$ 3.017 $\\
$ 0.93 $&$ 1.0 $&$ 1.245 $&$ 0.089 $&$ 0.3030 $&$ 0.333$&$ 13.250 $&$ 2.941 $\\
$ 0.94 $&$ 1.0 $&$ 1.245 $&$ 0.084 $&$ 0.3030 $&$ 0.333$&$ 13.872 $&$ 2.856 $\\
$ 0.95 $&$ 1.0 $&$ 1.245 $&$ 0.079 $&$ 0.3030 $&$ 0.333$&$ 14.513 $&$ 2.772 $\\
$ 0.96 $&$ 1.0 $&$ 1.245 $&$ 0.074 $&$ 0.3030 $&$ 0.332$&$ 15.173 $&$ 2.694 $\\
$ 0.97 $&$ 1.0 $&$ 1.245 $&$ 0.069 $&$ 0.3030 $&$ 0.332$&$ 15.850 $&$ 2.609 $\\
$ 0.98 $&$ 1.0 $&$ 1.245 $&$ 0.064 $&$ 0.3030 $&$ 0.331$&$ 16.544 $&$ 2.532 $\\
$ 0.99 $&$ 1.0 $&$ 1.245 $&$ 0.060 $&$ 0.3030 $&$ 0.331$&$ 17.253 $&$ 2.446 $\\
\hline
\end{tabular}
\end{table}
\end{small}
\indent For fixed values of $\s$, Table \ref{g1optatH0} provide bounds for $N(\s,H_0)$ of the shape \eqref{bnd-N@H0}.
In this case, the choice of $H$ is essential and we choose $H= H_0-10^{-6}$. As a consequence the ``main term'' is 
$\frac{10^{-6}}{2\pi d} (\log H_0) \log \Big( 1  + 10^{ 6}\mathcal{C}_1  (\log (kH_0))^{2\s}  (\log H_0)^{4(1-\s)} H_0^{\frac83(1-\s)} \Big)$ 
which becomes insignificant in comparison to $\frac{\mathcal{C}_2 (d, \eta, k, H, \mu,\s) }{2\pi d} (\log H_0)^2$, the term arising from the argument. 
We take $\alpha=0.324$, $\delta=0.3000$, and $k=1$ (as we did not find any other values giving better bounds).
The parameter $\eta$ is chosen to minimize $C_7(\eta,H)$, and then $\mu$ to minimize $\mu C_7(\eta,H)+C_8(k,\mu)$:
$\eta=0.2561 \ldots$ and $\mu=1.2453 \ldots$.
\begin{small}
\begin{table}[ht]
\caption{Bound \eqref{bnd-N@H0} with $k=1$}
\label{g1optatH0}
\begin{tabular}{|r|r|r|r|r|r|r|}
\hline
 $\s$  & $d$ & $\frac{1}{2\pi d}$ & $\mathcal{C}_1$ & $\frac{\mathcal{C}_2}{2\pi d}$ & $N(\s,H_0)\le$ \\
\hline
$ 0.60  $&$ 2.414$&$ 0.066 $&$ 2094.73 $&$ 0.893$&$ 520.28 $\\
$ 0.65  $&$ 3.621$&$ 0.044 $&$ 97986.60 $&$ 0.595$&$ 346.85 $\\
$ 0.70  $&$ 4.828$&$ 0.033 $&$ 4583580.34 $&$ 0.447$&$ 260.14 $\\
$ 0.75  $&$ 6.036$&$ 0.027 $&$ 214409007.32 $&$ 0.357$&$ 208.11 $\\
$ 0.80  $&$ 7.243$&$ 0.022 $&$ 10029544375.44 $&$ 0.298$&$ 173.42 $\\
$ 0.85  $&$ 8.450$&$ 0.019 $&$ 469158276689.92 $&$ 0.255$&$ 148.65 $\\
$ 0.86  $&$ 8.691$&$ 0.019 $&$ 1012341447042.27 $&$ 0.248$&$ 144.52 $\\
$ 0.87  $&$ 8.933$&$ 0.018 $&$ 2184412502812.95 $&$ 0.242$&$ 140.61 $\\
$ 0.88  $&$ 9.174$&$ 0.018 $&$ 4713486735514.76 $&$ 0.235$&$ 136.91 $\\
$ 0.89  $&$ 9.416$&$ 0.017 $&$ 10170678467214.40 $&$ 0.229$&$ 133.40 $\\
$ 0.90  $&$ 9.657$&$ 0.017 $&$ 21946110446020.33 $&$ 0.224$&$ 130.07 $\\
$ 0.91  $&$ 9.899$&$ 0.017 $&$ 47354929689448.17 $&$ 0.218$&$ 126.90 $\\
$ 0.92  $&$ 10.140$&$ 0.016 $&$ 102181631292174.11 $&$ 0.213$&$ 123.88 $\\
$ 0.93  $&$ 10.382$&$ 0.016 $&$ 220485720114084.42 $&$ 0.208$&$ 120.99 $\\
$ 0.94  $&$ 10.623$&$ 0.015 $&$ 475760194464125.94 $&$ 0.203$&$ 118.24 $\\
$ 0.95  $&$ 10.864$&$ 0.015 $&$ 1026586948666903.92 $&$ 0.199$&$ 115.62 $\\
$ 0.96  $&$ 11.106$&$ 0.015 $&$ 2215151194732183.30 $&$ 0.195$&$ 113.10 $\\
$ 0.97  $&$ 11.347$&$ 0.015 $&$ 4779814142285142.58 $&$ 0.190$&$ 110.70 $\\
$ 0.98  $&$ 11.589$&$ 0.014 $&$ 10313798574616601.14 $&$ 0.186$&$ 108.39 $\\
$ 0.99  $&$ 11.830$&$ 0.014 $&$ 22254932487167323.15 $&$ 0.183$&$ 106.18 $\\

\hline
\end{tabular}
\end{table}
\end{small}
%%%
\ 
%%%%
\end{document}